\documentclass[12pt]{elsarticle}

\usepackage{amsmath,amsthm,amsfonts,latexsym,amssymb,amscd,}
\usepackage[all, cmtip]{xy}
\usepackage{chemarr}
\usepackage{mathtools}
\usepackage{manfnt}
\pdfoutput=1

\newtheorem{thm}{Theorem}[section]

\newtheorem{lem}[thm]{Lemma}

\newtheorem{prop}[thm]{Proposition}
\theoremstyle{definition}
\newtheorem{defn}[thm]{Definition}

\theoremstyle{remark}
\newtheorem{rem}[thm]{Remark}

\DeclareMathOperator{\Con}{Con}

\DeclareMathOperator{\Pol}{Pol}

\DeclareMathOperator{\Cg}{Cg}

\DeclareMathOperator{\Sg}{Sg}

\def\nat{\mathbb{N}}
\def\A{\mathbb{A}}
\def\var{\mathcal{V}}
\def\Join{\bigvee}
\def\Meet{\bigwedge}
\def\Union{\bigcup}
\def\meet{\wedge}
\def\join{\vee}

\def\union{\cup}

\newcommand*{\concat}{\mathbin{\raisebox{0.9ex}{$\smallfrown$}}}

\begin{document}

\title{Higher Commutator Theory for Congruence Modular Varieties}

\author{Andrew Moorhead}
\address{Department of Mathematics\\
University of Colorado\\
Boulder, CO 80309-0395\\
USA}

\fntext[fn1]{This material is based upon work supported by
the National Science Foundation grant no.\ DMS 1500254}


\begin{abstract}
We develop the basic properties of the higher commutator for congruence modular varieties.
\end{abstract}

\maketitle

\section{Introduction}

This article develops some basic properties of a congruence lattice operation, called the higher commutator, for varieties of algebras that are congruence modular. The higher commutator is a higher arity generalization of the binary commutator, which was first defined in full generality in the seventies. While the binary commutator has a rich theory for congruence modular varieties, the theory of the higher arity commutator was poorly understood outside of the context of congruence permutability. 

We begin by discussing the evolution of centrality in Universal Algebra. Centrality is easily understood in groups as the commutativity of multiplication. Here, it plays an essential role in defining important group-theoretic notions such as abelianness, solvability, nilpotence, etc. Naturally, a systematic calculus to study centrality was developed.  For a group $G$ and $a, b \in G$, the group commutator of $a$ and $b$ is defined to be 

$$ [a, b] = a^{-1}b^{-1}ab.$$

Actually, one can go further and use group commutators to define a very useful operation on the lattice of normal subgroups of $G$. 

\begin{defn}\label{def:groupcom}
Suppose that $G$ is a group and $M$ and $N$ are normal subgroups of $G$. The group commutator of $M$ and $N$ is defined to be $$ [M,N]=  \Sg_G (\{ [m,n] : m \in M, n\in N \})$$

\end{defn} 

Suppose that $f: G \rightarrow H$ is a surjective homomorphism and $\{N_i: i \in I \}$ are normal subgroups of $G$. The following properties are easy consequences of Definition \ref{def:groupcom}, where $\meet$ and $\join$ denote the operations of meet and join in the lattice of normal subgroups of $G$:

\begin{enumerate}
\item $[M,N] \subset M \meet N$,
\item $[f(M), f(N)] = f([M,N])$,
\item $[M, N] = [N,M]$,
\item $[M , \Join_{i\in I} N_i] = \Join_{i \in I} [ M, N_i]$,
\item For any normal subgroup $K$ of $G$ contained in $M\meet N$, the elements of $M/K$ commute with $N / K$ if and only if $[M,N] \subset K$.

\end{enumerate}

Rings have an analogous commutator theory. For two ideals $I, J$ of a ring $R$ the commutator is $[I,J] = IJ -JI$. This operation satisfies the same basic properties as the commutator for groups and allows one to analogously define abelian, solvable and nilpotent rings. As it turns out, the notion of centrality and the existence of a well-behaved commutator operation is not an idiosyncrasy of groups or rings. In \cite{jdhsmith}, J.D.H. Smith defined a language-independent type of centrality that generalized the known examples. He then used this definition to show that any algebra belonging to a Mal'cev variety came equipped with a commutator as powerful as the commutator for groups or rings.

Hagemann and Hermann later extended the results of Smith to congruence modular varieties in \cite{hh}. The language-independent definition of centrality allows for language-independent definitions of abelianness and related notions such solvability and nilpotence. The existence of a robust commutator for a congruence modular variety means that these definitions are powerful and well-behaved, and provide an important tool to study the consequences of congruence modularity. For example, quotients of abelian algebras that belong to a modular variety are abelian, but this need not be true in general. 

The importance of these investigations was immediately apparent and the theory was rapidly developed, see \cite{fm} and \cite{gumm}. While the entirety of the theory is too broad for this introduction, we do mention an aspect related to nilpotence, because it is a prelude to the higher arity commutator. 

Roger Lyndon showed in \cite{lyndonnilgroup} that the equational theory of a nilpotent group is finitely based. Now, finite nilpotent groups are the product of their Sylow subgroups, so for finite groups Lyndon's result states that a group that is a product of $p$-groups has a finite basis for its equational theory. A result of Michael Vaughan-Lee, with an improvement due to Ralph Freese and Ralph McKenzie, generalizes this finite basis result to finite algebras generating a modular variety that are a product of prime power order nilpotent algebras, see \cite{fm}. Keith Kearnes showed in \cite{smallfreespec} that for a modular variety the algebras that are the product of prime power order nilpotent algebras are exactly the algebras that generate a variety with a small growth rate of the size of free algebras. 

Note that while for the variety of groups the condition of being a product of prime power order nilpotent algebras is equivalent to being nilpotent, this condition is in general stronger than nilpotence. This stronger condition is now known as \textbf{supernilpotence}, which is definable from the higher arity commutator that is the subject of our work, see \cite{aichmud}.

The definition of higher centrality was first introduced formally by Andrei Bulatov, see \cite{buldef}.  Bulatov was interested in counting the number of distinct polynomial clones on a finite set that contain a Mal'cev operation. Although this problem was solved in \cite{numfinalg} using other methods, higher commutators have found other important uses. In \cite{mayrbenz}, supernilpotence is shown to be an obstacle to a Mal'cev algebra having a natural duality. Also, as noted earlier, finite supernilpotent algebras that generate congruence permutable varieties must have a finitely based equational theory. 

Erhard Aichinger and Neboj{\v s}a Mundrinski developed the basic properties of the higher commutator for congruence permutable varieties, see \cite{aichmud}. In \cite{orsalrel}, Jakub Opr{\v s}al contributed to the properties of the higher commutator for Mal'cev varieties by developing a relational description that is similar to the original definition of centrality used by J.D.H. Smith, Hagemann and Hermann. 

The structure of this article is as follows: In Section 2 we introduce higher centrality, recall some important characterizations of congruence modularity and develop some notation. The main properties of the higher commutator are shown in Sections 3-5. In Section 6 we prove that for congruence modular varieties the higher commutator is equivalent to a higher commutator defined with a two term condition.

\section{Preliminaries}\label{prelims}

\subsection{Background}

We begin with the term condition definition of the $k$-ary commutator as introduced by Bulatov in \cite{buldef}. The following notation is used. Let $\A$ be an algebra with $\delta \in \Con(\A)$. A tuple will be written in bold: $\textbf{x} = (x_0,..., x_{n-1})$. The length of this tuple is denoted by $|\textbf{x}|$. For two tuples $\textbf{x}, \textbf{y}$ such that $|\textbf{x}| = |\textbf{y}|$ we write $\textbf{x} \equiv_\delta \textbf{y}$ to indicate that $x_i \equiv_\delta y_i$ for $0\leq i < |\textbf{x}|$, where $x_i \equiv_\delta y_i$ indicates that $\langle x,y \rangle \in \delta$.

\begin{defn}\label{def:bulcen}

Let $\A$ be an algebra, $k\in \nat_{\geq 2}$, and choose $\alpha_0,\dots, \alpha_{k-1}, \delta \in \Con(\A)$. We say that \textbf{$\alpha_0,\dots,\alpha_{k-2}$ centralize $\alpha_{k-1}$ modulo $\delta$} if for all $f\in \Pol(\A)$ and tuples $\textbf{a}_0,\textbf{b}_0, \dots, \textbf{a}_{k-1}, \textbf{b}_{k-1}$ from $\A$ such that

\begin{enumerate}
\item
$\textbf{a}_i \equiv_{\alpha_i} \textbf{b}_i$ for each $i \in k$

\item
If $f(\textbf{z}_0,\dots,\textbf{z}_{k-2},\textbf{a}_{k-1}) \equiv_\delta f(\textbf{z}_0,\dots,\textbf{z}_{k-2},\textbf{b}_{k-1})$ for all $(\textbf{z}_0,\dots, \textbf{z}_{k-2}) \in \{\textbf{a}_0, \textbf{b}_0 \}\times \dots \times \{\textbf{a}_{k-2}, \textbf{b}_{k-2} \} \setminus \{ (\textbf{b}_0,\dots, \textbf{b}_{k-2}) \}$

\end{enumerate}
we have that $$f(\textbf{b}_0, \dots, \textbf{b}_{k-2}, \textbf{a}_{k-1}) \equiv_\delta f(\textbf{b}_0, \dots, \textbf{b}_{k-2}, \textbf{b}_{k-1})$$

This condition is abbreviated as $C(\alpha_0, \dots, \alpha_{k-1}; \delta)$.
\end{defn}
It is easy to see that if for some collection $\{\delta_i: i\in I\} \subset \Con(\A)$ we have $C(\alpha_0, \dots, \alpha_{k-1}; \delta_i)$, then $C(\alpha_0, \dots , \alpha_{k-1}; \Meet_{i\in I} \delta_i )$. We therefore make the following
\begin{defn}\label{bulcomm}
Let $\A$ be an algebra, and let $\alpha_0, \dots, \alpha_{k-1} \in \Con(\A)$ for $k\geq 2$. The \textbf{$k$-ary commutator of $\alpha_0, \dots, \alpha_{k-1}$} is defined to be
$$ [\alpha_0,\dots, \alpha_{k-1}] = \Meet \{\delta: C(\alpha_0,\dots,\alpha_{k-1}; \delta) \}$$

\end{defn}

\noindent The following properties are immediate consequences of the definition:
\begin{enumerate}
\item
$[\alpha_0, \dots, \alpha_{k-1}] \leq \Meet_{0\leq i \leq k-1}\alpha_i$,
\item
For $\alpha_0 \leq \beta_0, \dots, \alpha_{k-1} \leq \beta_{k-1}$ in $\Con(\A)$, we have $[\alpha_0, \dots , \alpha_{k-1}] \leq [\beta_0, \dots , \beta_{k-1}]$ (Monotonicity),
\item
$[\alpha_0,\dots,\alpha_{k-1}] \leq [\alpha_1, \dots, \alpha_{k-1}]$.

\end{enumerate}

\noindent We will demonstrate the following additional properties of the higher commutator for a congruence modular variety $\var$, which are developed for the binary commutator in \cite{fm}:

\begin{enumerate}
\item[(4)] $ [\alpha_0,..., \alpha_{k-1}] = [\alpha_{\sigma(0)},..., \alpha_{\sigma(k-1)}]$ for any permutation of $\sigma$ of the congruences $\alpha_0,..., \alpha_{k-1}$ (Symmetry),

\item[(5)]$[\Join_{i\in I} \gamma_i, \alpha_1,..., \alpha_{k-1}] = \Join_{i\in I} [\gamma_i, \alpha_1,..., \alpha_{k-1}]$  (Additivity),

\item[(6)]  $[\alpha_0 , ..., \alpha_{k-1}] \join \pi = f^{-1}([f(\alpha_0 \join \pi),..., f(\alpha_{k-1} \join \pi))])$, where $f:\A \rightarrow \mathbb{B}$ is a surjective homomorphism with kernel $\pi$ (Homomorphism property),

\item[(9)] Kiss showed in \cite{threeremarks} that for congruence modular varieties the binary commutator is equivalent to a binary commutator defined with a two term condition. This is true for the higher commutator also.

\end{enumerate}

\subsection{Day Terms and Shifting}\label{DayTerms}
The following classical results about congruence modularity are needed. For proofs see \cite{dayterms}, \cite{gumm} and \cite{fm}. 
\begin{prop}[Day Terms]\label{prop:day}
A variety $\var$ is congruence modular if and only if there exist term operations $m_e(x,y,z,u)$ for $e\in n+1$ satisfying the following identities:

\begin{enumerate}
\item $m_e(x,y,y,x) \approx x$ for each $0\leq e \leq n$,
\item $m_0(x,y,z,u) \approx x$,
\item $m_n(x,y,z,u) \approx u$,
\item $m_e(x,x,u,u) \approx m_{e+1}(x,x,u,u) $ for even $e$, and
\item $m_e(x,y,y,u) \approx m_{e+1}(x,y,y,u)$ for odd $e$.

\end{enumerate}
\end{prop}

\begin{prop}[Lemma 2.3 of \cite{fm}]\label{prop:shift}
Let $\var$ be a variety with Day terms $m_e$ for $e\in n+1$. Take $\delta \in \Con(\A)$ and assume $\langle b,d \rangle  \in \delta$. For a tuple $\langle a,c \rangle \in A^2$ the following are equivalent:

\begin{enumerate}
\item $\langle a,c \rangle \in \delta$,
\item $\langle m_e(a,a,c,c) , m_e(a,b,d,c) \rangle \in \delta$ for all $e\in n+1$.
\end{enumerate}
\end{prop}

\begin{lem}[The Shifting Lemma] \label{prop:shiftlemma}
Let $\var$ be a congruence modular variety, and take $\A \in \var$. Take $\theta_1 , \theta_2\in \Con(\A)$ and $\gamma \geq \theta_1 \meet \theta_2$. Suppose $a,b,c,d \in A$ are such that $\langle a,b \rangle, \langle c,d \rangle \in \theta_1$, $\langle a,c \rangle, \langle b,d \rangle \in \theta_2$ and $\langle b,d \rangle \in \gamma$. Then $\langle a,c \rangle \in \gamma$. Pictorially,

$$\xymatrix{
a   \ar@{-}^{\theta_1}[r]  \ar@{-}_{\theta_2}[d]&   b\ar@{-}[d] \ar@{-}@/^/ ^\gamma[d]\\
c \ar@{-}[r]   &   d
}
\textnormal{            \emph{implies}                 }
\xymatrix{
a   \ar@{-}^{\theta_1}[r]  \ar@{-}_{\theta_2}[d] \ar@{-}@/^/ ^\gamma[d]&   b\ar@{-}[d] \ar@{-}@/^/ ^\gamma[d]\\
c \ar@{-}[r]   &   d
}
$$

\end{lem}

\subsection{Matrices and Centralization}

Take $\A \in \var$ and $\theta_0, \theta_1 \in \Con(\A)$. The development of the binary commutator in \cite{fm} relies on a so-called term condition that can be defined with respect to a subalgebra of $\A^4$, the subalgebra of $(\theta_0, \theta_1)$-matrices. We will now generalize these ideas to the higher commutator. To motivate the definitions, we state them for the binary commutator.

\begin{defn}[Binary]\label{FMmat}
Take $\A \in \var$, and $\theta_0, \theta_1, \in \Con(\A)$. Define 
$$ M(\theta_0, \theta_1) = \left\{ \left[  \begin{array}{cc} 
	t(\textbf{a}_0,\textbf{a}_1)  & t(\textbf{a}_0,\textbf{b}_1)\\
	t(\textbf{b}_0,\textbf{a}_1)& t(\textbf{b}_0,\textbf{b}_1)\\
				\end{array}\right] : t\in \Pol(\A), \textbf{a}_0 \equiv_{\theta_0} \textbf{b}_0, \textbf{a}_1 \equiv_{\theta_1} \textbf{b}_1 \right\} $$
\end{defn}
It is readily seen that $M(\theta_0, \theta_1)$ is a subalgebra of $\A^4$, with a generating set of the form

$$
\left\{ \left[  \begin{array}{cc} 
	x  & x\\
	y & y\\
				\end{array}\right] : x\equiv_{\theta_0} y
\right\} \Union \left\{ \left[  \begin{array}{cc} 
	x  & y\\
	x & y\\
				\end{array}\right]  : x\equiv_{\theta_1} y
\right\}
$$
The notion of centrality given in Definition \ref{def:bulcen} with congruences $\theta_0, \theta_1, \delta$ is expressible as a condition on $(\theta_0, \theta_1)$-matrices. This is shown in Figure \ref{fig:bincen}, where for $\delta \in \Con(\A)$ the implications depicted hold for all  $$\left[  \begin{array}{cc} 
	t(\textbf{a}_0,\textbf{a}_1)  & t(\textbf{a}_0,\textbf{b}_1)\\
	t(\textbf{b}_0,\textbf{a}_1)& t(\textbf{b}_0,\textbf{b}_1)\\
				\end{array}\right] \in M(\theta_0 , \theta_1) .$$
 
\begin{figure}[!ht]

\includegraphics{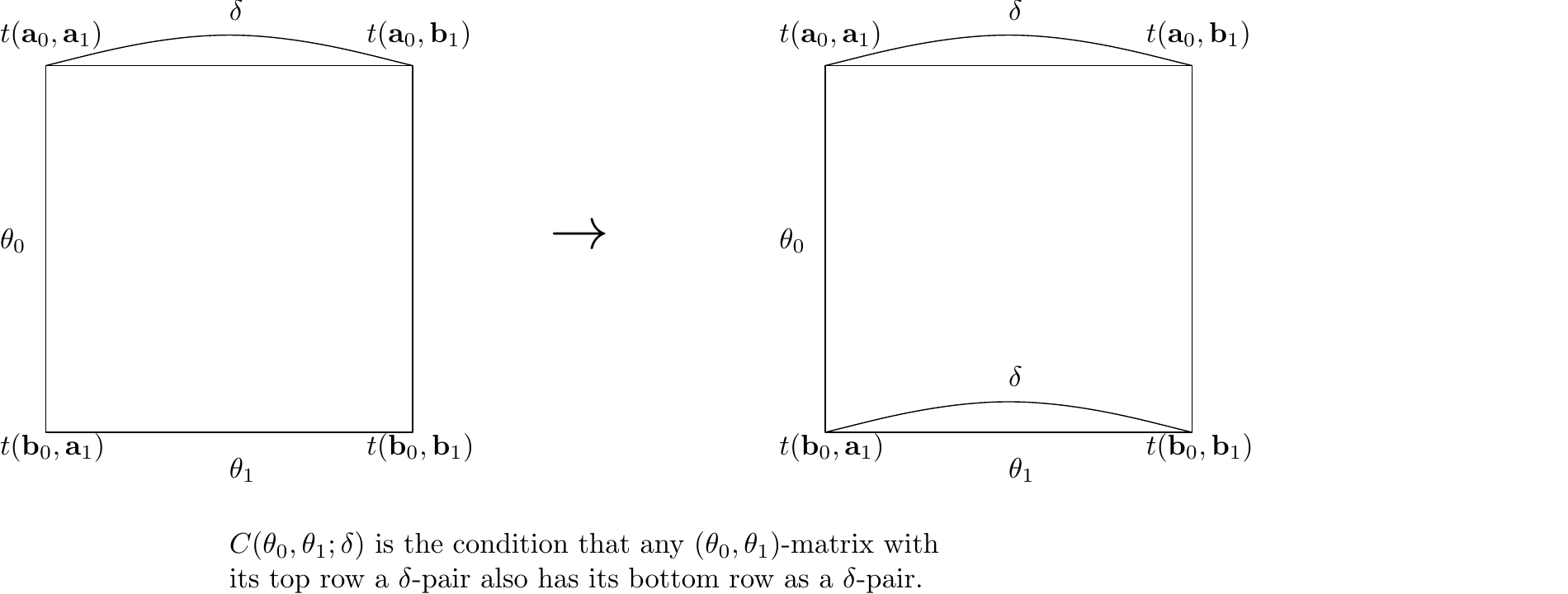}

\vspace{5mm}%

\includegraphics{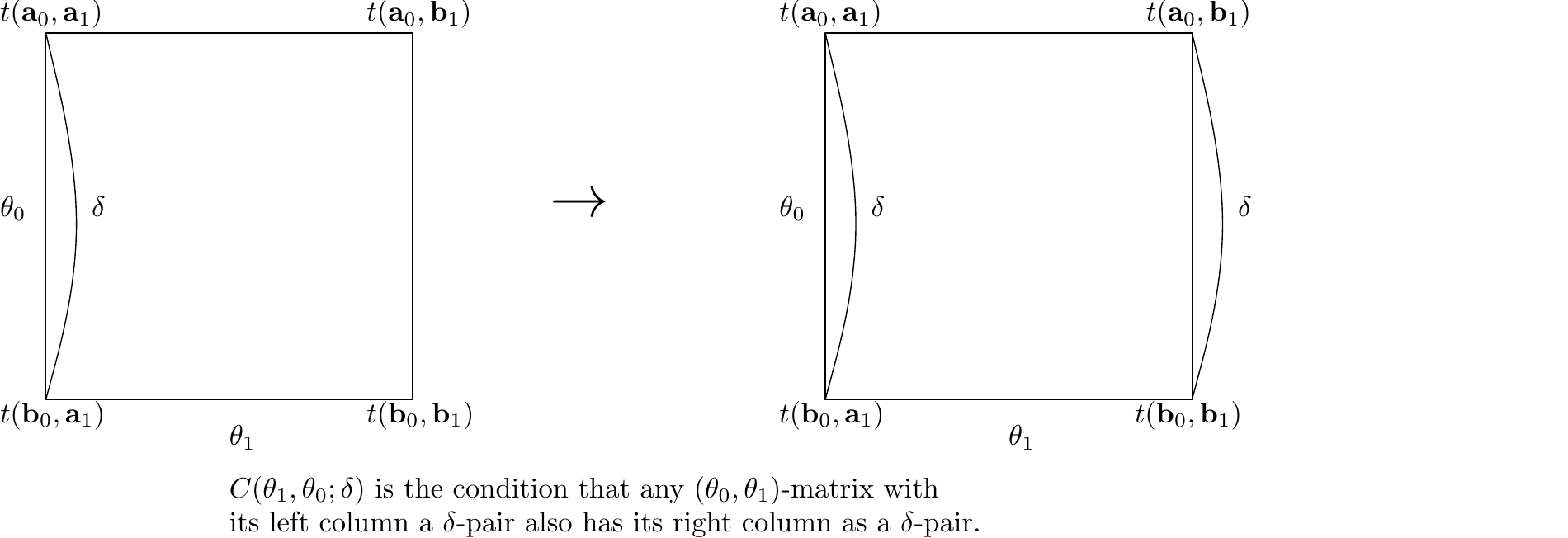}

\caption{Binary centrality}\label{fig:bincen}
\end{figure}

It is easy to generalize the idea of matrices to three dimensions. For congruences $\theta_0, \theta_1, \theta_2, \delta$ of an algebra $\A$, the condition $C(\theta_1, \theta_2, \theta_0; \delta)$ is equivalent to the implication depicted in Figure \ref{fig:cubecen} for all $t\in \Pol(\A)$ and $\textbf{a}_0 \equiv_{\theta_0} \textbf{b}_0, \textbf{a}_1 \equiv_{\theta_1} \textbf{b}_1$, $\textbf{a}_2 \equiv_{\theta_2} \textbf{b}_2$.

\begin{figure}[!ht]

\includegraphics{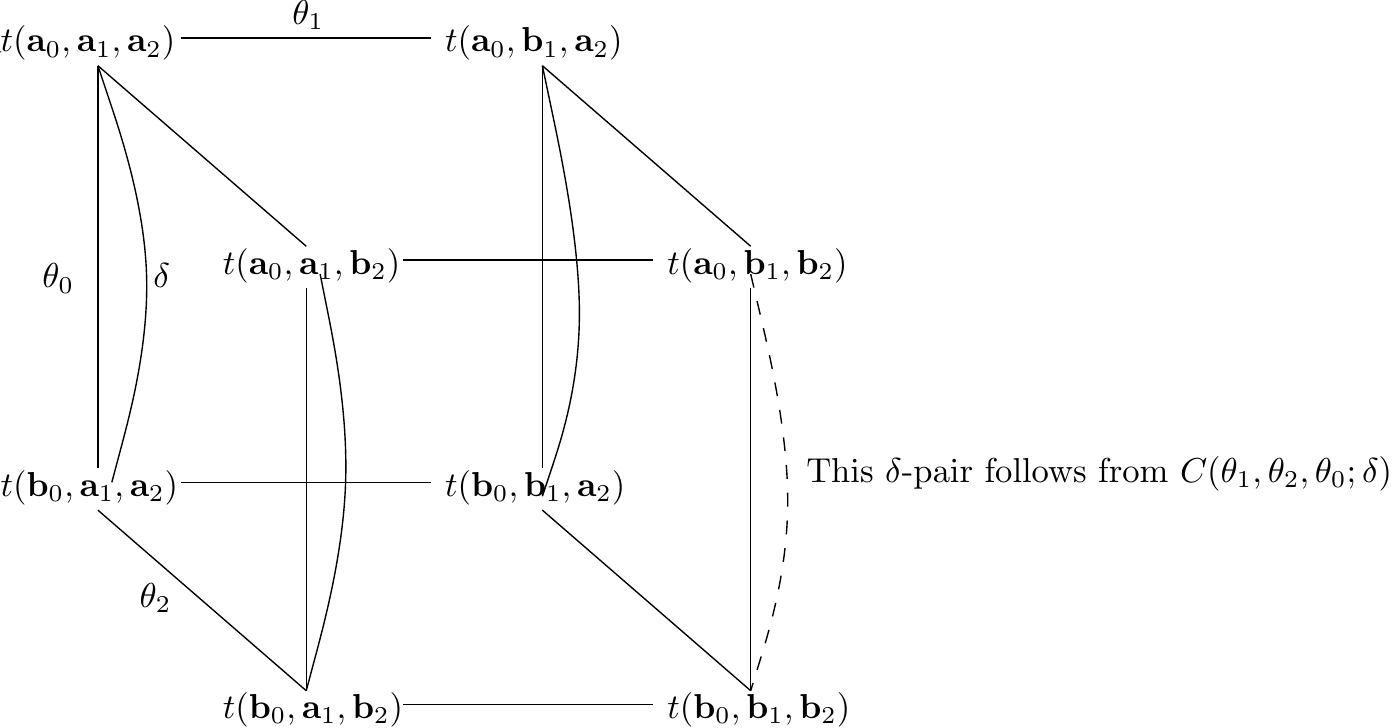}

\caption{Ternary Centrality}\label{fig:cubecen}

\end{figure}

The main arguments in this paper are essentially combinatorial and rely on isolating certain squares and lines in matrices. In the case of the matrix shown in Figure \ref{fig:cubecen}, we identify the squares shown in Figure \ref{fig:pivsquare1}, which we label as $(0,1)$-supporting and pivot squares (see Definition \ref{supportpiv}). Notice that both squares are $(\theta_0, \theta_1)$-matrices, where the supporting square corresponds to the polynomial $t(\textbf{z}_0, \textbf{z}_1, \textbf{a}_2)$ and the pivot square corresponds to the polynomial $t(\textbf{z}_0, \textbf{z}_1, \textbf{b}_2)$.

\begin{figure}[!ht]

\includegraphics{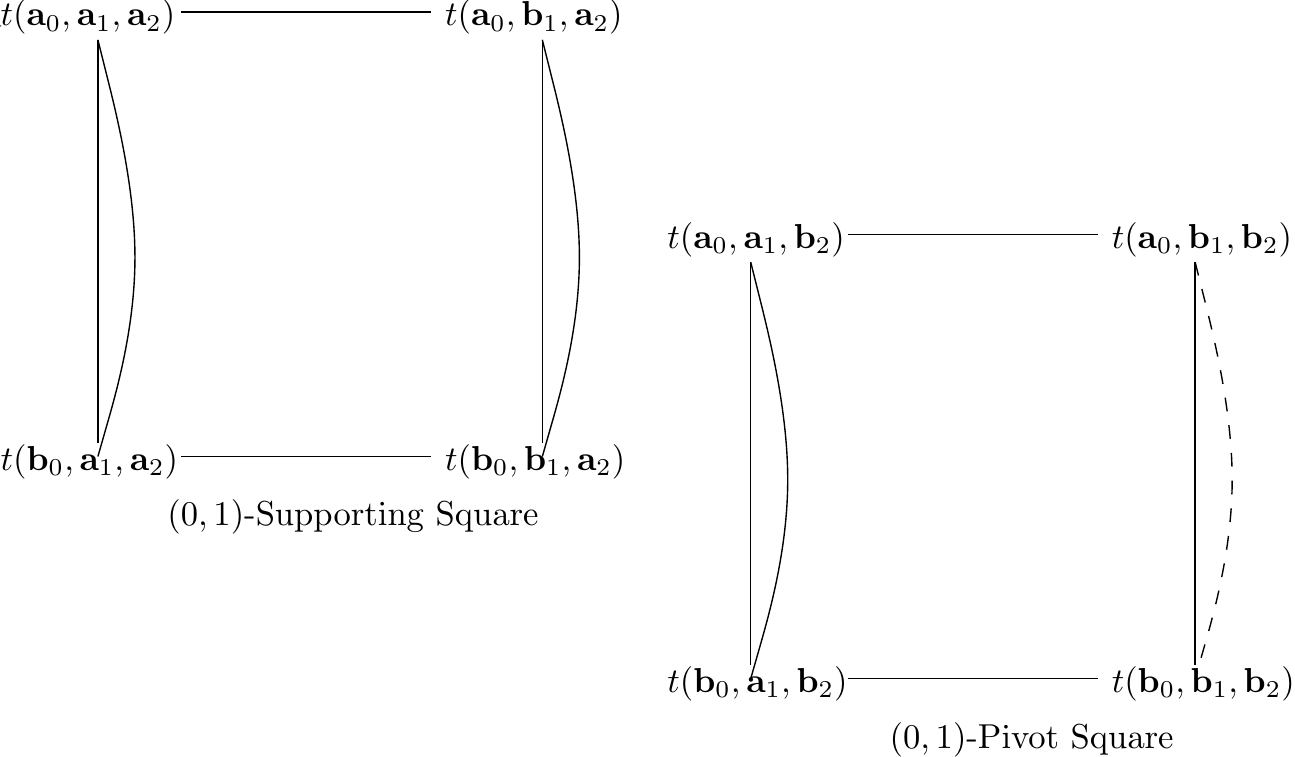}

\vspace{5mm}

\includegraphics{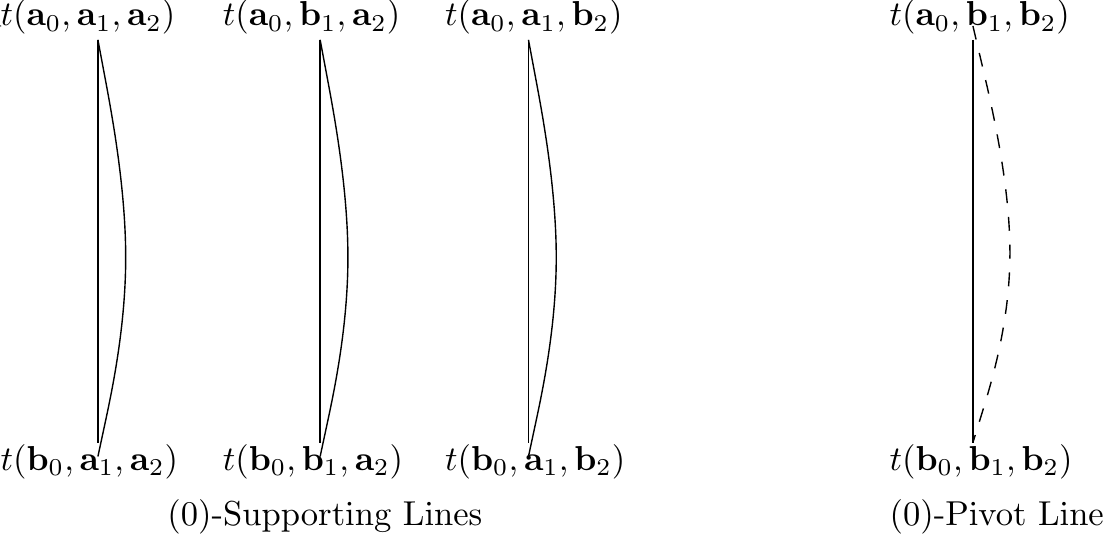}

\caption{Squares and Lines}\label{fig:pivsquare1}

\end{figure}

We also identify the lines shown in Figure \ref{fig:pivsquare1}, which are labeled as either a $(0)$-supporting line or a $(0)$-pivot line (see Definition \ref{supportpiv}). Notice that each line corresponds to a polynomial $s(\textbf{z}_0)= t(\textbf{z}_0, \textbf{x}_1, \textbf{x}_2)$, where $\textbf{x}_1 \in \{\textbf{a}_1, \textbf{b}_1 \}$ and $\textbf{x}_2 \in \{\textbf{a}_2, \textbf{b}_2 \}$. Notice that $C(\theta_1,\theta_2, \theta_0; \delta)$ is equivalent to the statement that if every $(0)$-supporting line of such a matrix is a $\delta$ pair, then the $(0)$-pivot line is a $\delta$-pair.

We therefore require for a sequence of congruences $(\theta_0, \dots, \theta_{k-1})$ the notion of a matrix, as well as a notation to identify a matrice's supporting and pivot squares and lines. 

 \begin{defn}\label{tmatrix}

Let $T = (\theta_0,\dots, \theta_{k-1}) \in \Con(\A)^k$ be a sequence of congruences of $\A$. A pair $\tau = (t, \mathcal{P})$ is called a \textbf{$T$-matrix label} if 
\begin{enumerate}
\item
$t = t(\textbf{z}_0, \dots , \textbf{z}_{k-1}) \in \Pol(\A)$

\item
$\mathcal{P} = (P_0,\dots , P_{k-1})$ is a sequence of pairs $P_i=(\textbf{a}_i, \textbf{b}_i)$ such that $\textbf{a}_i\equiv_{\theta_i} \textbf{b}_i$
\end{enumerate}

\end{defn}

Let $\tau = (t(\bf{z}_0,\dots, \bf{z}_{k-1}),\mathcal{P})$ be a $T$-matrix label. From the above examples, we see that $\tau$ can be used to construct a $k$-dimensional cube whose vertices correspond to evaluating each variable tuple $\textbf{z}_i$ in $t$ at one of the tuples belonging to $P_i$. We also need to identify the squares and lines of this matrix, which are in fact $2$ and $1$-dimensional matrices.  As in the above examples, these objects correspond to the evaluation of some of the $\textbf{z}_i$ at tuples in $\mathcal{P}$. We introduce notation to identify which of the $\textbf{z}_i$ in $t(\textbf{z}_0, \dots, \textbf{z}_{k-1})$ are being evaluated and which variable tuples $\textbf{z}_i$ remain free.

Let $S \subset k$. Denote by $T_S$ the subsequence $(\theta_{i_1},\dots,\theta_{i_s})$ of congruences from $T$ that is indexed by $S$. For a function $f\in 2^{k \setminus S}$ let $\tau_f = (t\vert_f, \mathcal{P}_S)$ be the $T_S$-matrix label such that
\begin{enumerate}
\item
$t\vert_f(\textbf{z}_{i_1},\dots, \textbf{z}_{i_s})= t(\textbf{x}_1,\dots,\textbf{x}_k)$ with
\begin{enumerate}
\item $(\textbf{z}_{i_1}, \dots, \textbf{z}_{i_s})$ is the collection of variable tuples indexed by $S$
\item $\textbf{x}_i = \textbf{z}_i$ if $i\in  S$
\item $\textbf{x}_i = \textbf{a}_i$ if $f(i)=0$
\item $\textbf{x}_i = \textbf{b}_i$ if $f(i)=1$
\end{enumerate}
\item 
$\mathcal{P}_S$ is the subsequence $(P_{i_1}, \dots, P_{i_s})$ of pairs of tuples from $\mathcal{P}$ that is indexed by $S$.
\end{enumerate}

Notice that if $S = \emptyset $ then each $\tau_f$ specifies a way in which to evaluate each tuple $\textbf{z}_i$ at either $\textbf{a}_i$ or $\textbf{b}_i$. As we will see, these are vertices of the matrices which we now define.

\begin{defn}
Choose $k \geq 1$. Let $T= (\theta_0,\dots, \theta_{k-1})$ be a sequence of congruences of $\A$. Let $\tau = (t, \mathcal{P})$ be a $T$-matrix label. A \textbf{$T$-matrix} is an element 
$$m \in \prod_{f\in 2^k} \A = \A^{2^k}$$ 
such that $m_f = t\vert_f$ for all $f\in 2^{k}$. We say in this case that $m$ is \textbf{labeled} by $\tau$. Denote by $M(T)$ the collection of all $T$-matrices.
\end{defn}

\begin{rem}
If $T=(\theta_1,\dots, \theta_k)$ is a sequence of congruences of $\A$ then $M(T)$ and $M(\theta_1,\dots, \theta_k)$ denote the same collection. 
\end{rem}

If we consider the set $k$ as a set of coordinates, the set of functions $2^{k}$ can be viewed as a $k$-dimensional cube, where $f$ is connected to $g$ by an edge if $f(i)=g(i)$ for all $i \in k \setminus \{ j \}$ for some coordinate $j$. Each $T$-matrix $m$ labeled by $\tau$ is therefore a $k$-dimensional cube, with a vertex $m_f$ for each $f\in 2^k$. Moreover, if $m_f$ and $m_g$ are connected by an edge where $f(i)=g(i)$ for all $i\in k \setminus \{j\}$ for some coordinate $j$, then $m_f \equiv_{\theta_j} m_g$. 

As noted in the case of the binary commutator, the collection of $(\alpha, \beta)$-matrices is a subalgebra of $\A^4$ and is generated by those $m\in M(\alpha, \beta)$ that are constant across rows or columns. These facts easily generalize to the collection of $T$-matrices.

\begin{lem}\label{tmatrixlem}
Let $T= (\theta_0,\dots, \theta_{k-1})$ be a sequence of congruences of an algebra $\A$. The collection $M(T)$ forms a subalgebra of $\A^{2^k}$, and is generated by those matrices $m \in M(T)$ that depend only on one coordinate.

\end{lem}

We now define the ideas of a cross-section square and a cross-section line.  Let $T= (\theta_0\dots,\theta_{k-1})$ be a sequence of congruences and $m \in M(T)$ be labeled by $\tau = (t, \mathcal{P})$. Choose two coordinates $j,l \in k$ with $j \neq l$. For $f^*\in  2 ^{k \setminus\{j,l\}}$ let $m_{f^*} \in M(\theta_j, \theta_l)$ be the $(\theta_j, \theta_l)$-matrix labeled by $\tau_{f^*}$. We call $m_{f^*}$ the \textbf{$(j,l)$-cross-section square} of $m$ at $f^*$. Similarly, for a coordinate $j\in k$ and $f\in  2 ^{k\setminus\{j\}}$ let $m_f \in M(\theta_j)$ be the $(\theta_j)$-matrix labeled by $\tau_f$. We call $m_f$ in this case the \textbf{$(j)$-cross-section line} of $m$ at $f$. 

A typical $(j,l)$-cross-section square $m_{f^*}$ will be displayed as 
$$m_{f^*}= \left[  \begin{array}{cc} 
	t_{f^*}(\textbf{a}_j, \textbf{a}_l) & t_{f^*}(\textbf{a}_j, \textbf{b}_l)\\
	t_{f^*}(\textbf{b}_j, \textbf{a}_l)& t_{f^*}(\textbf{b}_j, \textbf{b}_l)\\
				\end{array}
\right] = 
\left[  \begin{array}{cc} 
	r_{f^*} & s_{f^*}\\
	u_{f^*}& v_{f^*}\\
				\end{array}
\right]$$
and a typical $(j)$ or $(l)$-cross-section line of $m$ is a column or row, respectively, of such a square.

We set 
$$ S(m;j,l) = \{m_{f^*}
: f^* \in  2^{k \setminus\{j,l\} }
\}
\textnormal{ and }$$ 
$$ L(m;j) = \{m_f
: f \in  2^{k \setminus\{j\} }
\}
$$ 
to be the collections of all $(j,l)$-cross-section squares and $(j)$-cross-section lines of $m$, respectively.

\begin{defn}\label{supportpiv}
Let $T = (\theta_0, \dots, \theta_{k-1}) \in \Con(\A)^k$, and take $m\in M(T)$. Choose $j,l \in k$ such that $j\neq l$. Let $\textbf{jl} \in 2^{k \setminus \{j,l\}}$, $\textbf{j} \in 2^{k \setminus \{j\}}$ and $\textbf{1} \in 2^k$ be the constant functions that take value $1$ on their respective domains. We call the $(j,l)$-cross-section square of $m$ at $\textbf{jl}$ the \textbf{$(j,l)$-pivot square}. All other $(j,l)$ cross-section squares of $m$ will be called \textbf{$(j,l)$-supporting squares}. Similarly, we call the $(j)$ cross-section square of $m$ at $\textbf{j}$ the \textbf{$(j)$-pivot line}, and all other $(j)$ cross-section lines will be called \textbf{$(j)$-supporting lines}.
\end{defn}

We now reformulate Definition \ref{def:bulcen} with respect to these definitions. 

\begin{defn}\label{def:cen}
We say that \textbf{$T$ is centralized at $j$ modulo $\delta$} if the following property holds for all $T$-matrices $m \in M(T)$:
\begin{enumerate}
\item[(*)] If every $(j)$-supporting line of $m$ is a $\delta$-pair, then the $(j)$-pivot line of $m$ is a $\delta$-pair. 

\end{enumerate}

We abbreviate this property $C(T; j; \delta)$.
\end{defn}

\begin{defn}\label{def:commutator}
We define $[T]_j = \bigwedge \{\delta : C(T; j; \delta) \}$
\end{defn}

\begin{rem}\label{diffcomm}
Notice that $[T]_j = [\theta_{i_0}, \dots, \theta_{i_{k-2}}, \theta_j]$ for any permutation of the $k-1$ congruences that are not $\theta_j$, where the left side is given by Definition \ref{def:commutator} and the right is given by Definition \ref{bulcomm}. 
\end{rem}

We conclude this chapter with a general picture of the $(j,l)$-supporting and pivot squares of a $T$-matrix $m$ labeled by some $\tau = (t, \mathcal{P})$, a $T$-matrix label for a sequence of congruences $T= (\theta_0, \dots, \theta_{k-1})$. The conditions $C(T;j;\delta)$ and $C(T;l;\delta)$ are shown in Figure \ref{fig:supppivsquares}, respectively.

\begin{figure}

\includegraphics{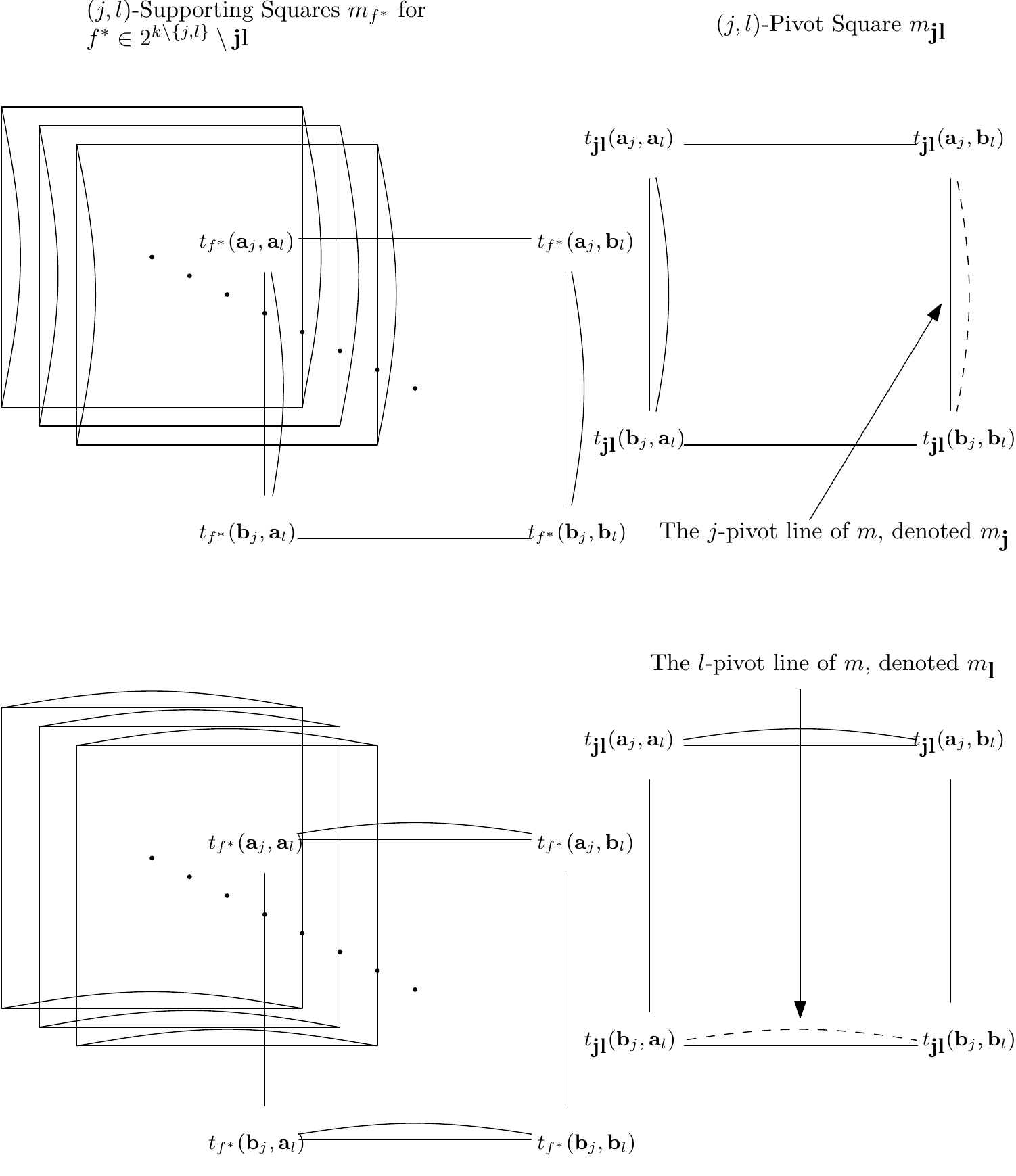}

\caption{Higher Centrality, Squares and Lines}\label{fig:supppivsquares}

\end{figure}

\section{Symmetry of Higher Commutator}

For the remainder of this document a variety $\var$ is assumed to be congruence modular. In this section we will show that the commutator of Definition \ref{bulcomm} is symmetric. We fix $\A \in \var$, with $\var$ a congruence modular variety with Day terms $m_e$ for $e \in n+1$. For $k \geq 2$ let $T = (\theta_0, \dots, \theta_{k-1}) \in \Con(\A)^k$ be a sequence of congruences of $\A$. We wish to show that $[\theta_0, \dots, \theta_{k-1}] = [ \theta_{\sigma(0)}, \dots, \theta_{\sigma(k-1)}]$ for any permutation $\sigma$ of the elements of $k$. By Remark \ref{diffcomm} it will suffice to show that $[T]_j = [T]_l$ for all $j,l \in k$. This will imply that $[\theta_0, \dots, \theta_{k-1}] = [ \theta_{\sigma(0)}, \dots, \theta_{\sigma(k-1)}] = [T]_j = [T]_l$ for all permutations $\sigma$ of $k$ and all $j, l \in k$. 

We begin with the following

\begin{lem}\label{lem:newop}
Let $\var$ be a congruence modular variety with Day terms $m_e$ for $e \in n+1$, and let $\A \in \var$. Let $T=(\theta_0,\dots \theta_{k-1}) \in \Con(\A)^k$. For each choice of $j,l\in k$ such that $j \neq l$ and $e \in n+1$ there is a map $R^e_{j,l}:M(T) \rightarrow M(T)$ with the following properties: 

\begin{enumerate}
\item
If $h  \in M(T)$ has the set of $(j,l)$-cross-section squares $$ S(h;j,l) = \left\{ h_{f^*} = \left[  \begin{array}{cc} 
	r_{f^*}  & s_{f^*}\\
	u_{f^*} & v_{f^*}\\
				\end{array}
\right]
: f^* \in  2 ^{k \setminus\{j,l\} }\right\}
$$
then $R^e_{j,l}(h)$ has the set of $(j,l)$-cross-section squares $S(R^e_{j,l}(h);j,l) = $
$$ \left\{R^e_{j,l}(m)_{f^*}=  \left[  \begin{array}{cc} 
	s_{f^*}  & s_{f^*}\\
	m_e(s_{f^*}, r_{f^*}, u_{f^*}, v_{f^*}) & m_e(s_{f^*}, s_{f^*}, v_{f^*}, v_{f^*})\\
				\end{array}
\right]
: f^* \in  2 ^{k \setminus\{j,l\} }
\right\}
$$

\item
If every $(j)$-supporting line of $h$ is a $\delta$-pair, then every $(l)$-supporting line of $R^e_{j,l}(h)$ is a $\delta$-pair. 

\item 
Suppose the $(j)$-supporting line belonging to the $(j,l)$-pivot square of $h$ is a $\delta$-pair. The $(j)$-pivot line of $h$ is a $\delta$-pair if and only if the $(l)$-pivot line of $R^e_{j,l}(h)$ is a $\delta$-pair for all $e\in n+1$.

\end{enumerate}

The map $R^e_{j,l}$ will be called \textbf{the $\bf{e}$th shift rotation at $(j,l)$}. 
\end{lem}

\begin{proof}

Let $h\in M(T)$ be labeled by $\tau = (t, \mathcal{P})$, where $t = t(\textbf{z}_0, \dots, \textbf{z}_{k-1} )$ and $\mathcal{P} = (P_0, \dots, P_{k-1})$ with $P_i = (\textbf{a}_i, \textbf{b}_i)$. Fix $j,l \in k$ with $j \neq l$ and take $e \in n+1$. Let 
$$t^e_{j,l}(\textbf{y}_0, ..., \textbf{y}_{k-1} ) =m_e(t_0, t_1, t_2, t_3) $$
where
\begin{align*}
t_0 =& t(\textbf{y}_0,\dots , \textbf{y}_j^0,\dots , \textbf{y}_l^0,\dots ,\textbf{y}_{k-1})\\
t_1 =& t(\textbf{y}_0,\dots , \textbf{y}_j^1,\dots , \textbf{y}_l^1,\dots ,\textbf{y}_{k-1})\\
t_2 =& t(\textbf{y}_0,\dots , \textbf{y}_j^2,\dots , \textbf{y}_l^2,\dots ,\textbf{y}_{k-1})\\
t_3 =& t(\textbf{y}_0,\dots , \textbf{y}_j^3,\dots , \textbf{y}_l^3,\dots ,\textbf{y}_{k-1})\
\end{align*}
and $\textbf{y}_j = \textbf{y}_j^0 \concat \textbf{y}_j^1 \concat \textbf{y}_j^2 \concat \textbf{y}_j^3 $,
$\textbf{y}_l = \textbf{y}_l^0 \concat \textbf{y}_l^1 \concat \textbf{y}_l^2 \concat \textbf{y}_l^3 $
are concatenations. 

For each $i \in k$ , define a pair of tuples $P_i^e = (\textbf{a}_i', \textbf{b}_i' ) $ as follows:
\begin{enumerate}
\item
$P_i^e= P_i $ if $i \neq j, l$

\item 
$P_j^e =  (\textbf{a}_j', \textbf{b}_j' ) = (  (\textbf{a}_j \concat \textbf{b}_j \concat \textbf{b}_j \concat \textbf{a}_j) , (\textbf{a}_j \concat \textbf{a}_j \concat \textbf{b}_j \concat \textbf{b}_j ) ) $

\item 
$P_l^e =  (\textbf{a}_l', \textbf{b}_l') = ( (\textbf{b}_l \concat \textbf{a}_l \concat \textbf{a}_l \concat \textbf{b}_l) ,  (\textbf{b}_l \concat \textbf{b}_l \concat \textbf{b}_l \concat \textbf{b}_l )) $

\end{enumerate}
 
 Let $\mathcal{P}^e_{j,l}= (P^e_0,\dots, P^e_{k-1})$, and set $\tau^e_{j,l} = (t^e_{j,l},\mathcal{P}_{j,l}^e)$. Define $R^e_{j,l}(h) \in M(T)$ to be the $T$-matrix labeled by $\tau^e_{j,l}$.

We now compute $S(R^e_{j,l}(h);j,l)$, the set of $(j,l)$ cross-section squares of $R^e_{j,l}(h)$.  Take $f^* \in  2 ^{k \setminus\{j,l\} }$. Consider the $(j,l)$ cross-section square of $h$ at $f^*$: 
$$
h_{f^*}= \left[  \begin{array}{cc} 
	r_{f^*}  & s_{f^*}\\
	u_{f^*} & v_{f^*}\\
				\end{array}
\right]
$$
By the definitions given above we therefore compute 

\begin{align*}
R^e_{j,l}(h)_{f^*}=& \left[  \begin{array}{cc} 
	(t^e_{j,l})_{f^*}(\textbf{a}_j', \textbf{a}_l'  ) & (t^e_{j,l})_{f^*}(\textbf{a}_j', \textbf{b}_l'  ) \\
					(t^e_{j,l})_{f^*}(\textbf{b}_j', \textbf{a}_l'  )  & (t^e_{j,l})_{f^*}(\textbf{b}_j', \textbf{b}_l'  ) \\
				\end{array}
\right]\\
=&
\left[  \begin{array}{cc} 
	m_e(s_{f^*}, u_{f^*}, u_{f^*}, s_{f^*}) & m_e(s_{f^*}, v_{f^*}, v_{f^*}, s_{f^*})\\
					m_e(s_{f^*}, r_{f^*}, u_{f^*}, v_{f^*}) & m_e(s_{f^*}, s_{f^*}, v_{f^*}, v_{f^*})\\
				\end{array}
\right]\\
=&
\left[  \begin{array}{cc} 
	s_{f^*} & s_{f^*}\\
					m_e(s_{f^*}, r_{f^*}, u_{f^*}, v_{f^*}) & m_e(s_{f^*}, s_{f^*}, v_{f^*}, v_{f^*})\\
				\end{array}
\right]
\end{align*}
where the final equality follows from identity (1) in Proposition \ref{prop:day}. This proves (1) of the lemma. 

\indent We now prove (2) and (3). A picture is given in Figure  \ref{fig:shiftrot}, where a typical $(j,l)$-supporting square and the $(j,l)$-pivot square are shown for both $h$ and $R^e_{j,l}(h)$. Supporting lines are drawn in bold.

\begin{figure}[!ht]

\includegraphics{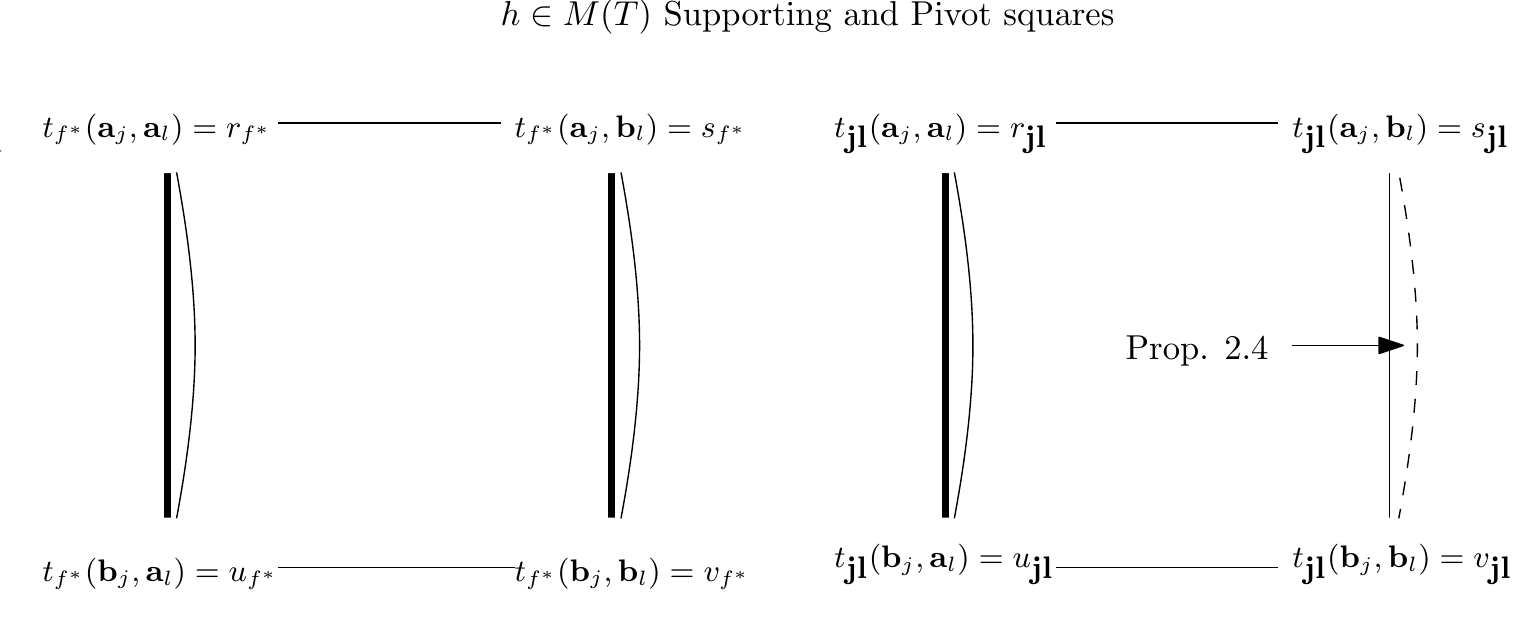}

\vspace{5mm}

\includegraphics{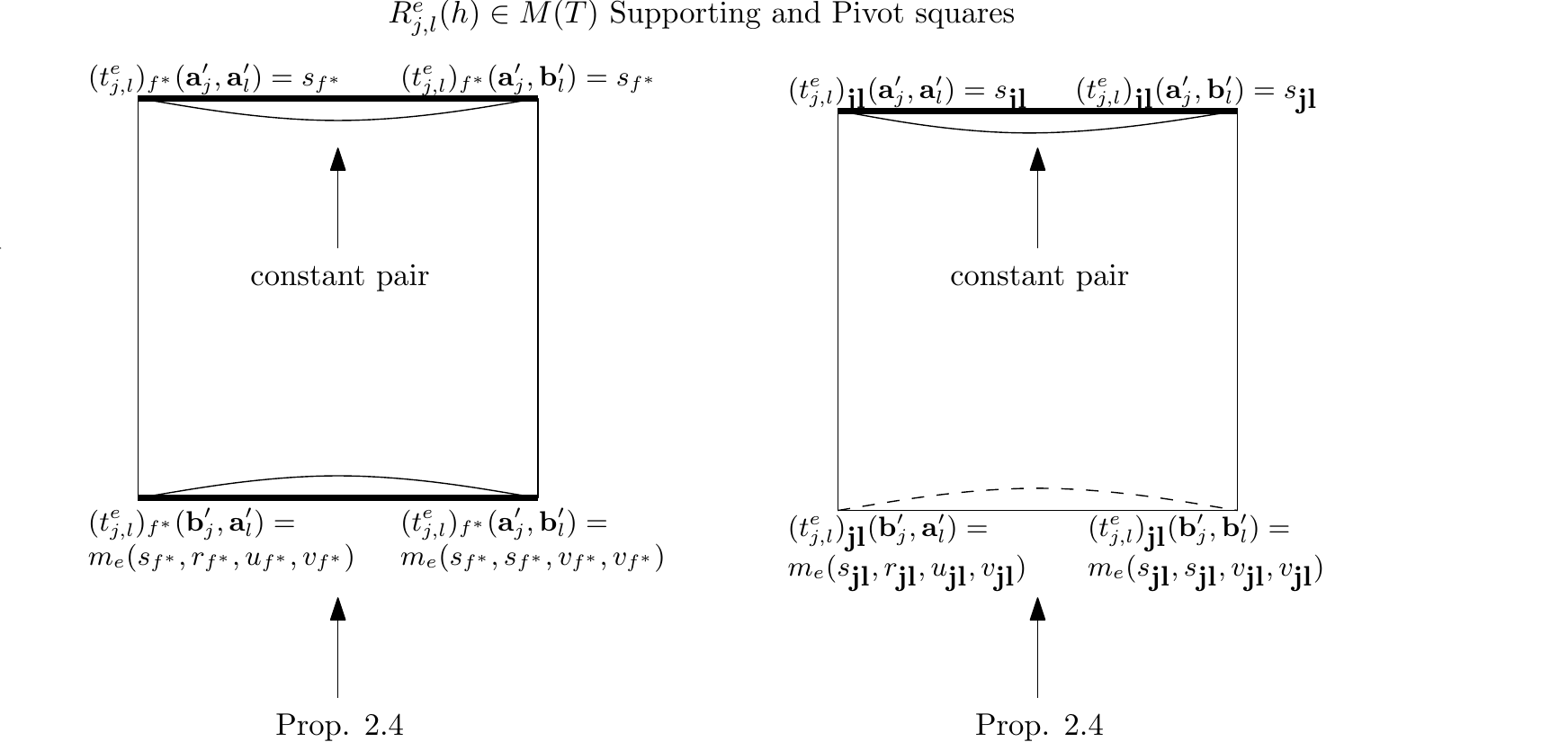}

\caption{Shift Rotations}\label{fig:shiftrot}

\end{figure}

Indeed, any constant pair $\langle s,s \rangle$  is a $\delta$-pair, so the top row of any $(j,l)$-cross-section square of $R^e_{j,l}(h)$ is a $\delta$-pair. That the other $(l)$-supporting lines of $R^e_{j,l}(h)$ are $\delta$-pairs follows from Proposition \ref{prop:shift}. Finally, Proposition $\ref{prop:shift}$ shows that the $(j)$-pivot line of $h$ is a $\delta$-pair if and only if for every $e\in n+1$ the $(l)$-pivot line of $R^e_{j,l}(h)$ is a $\delta$-pair, which is indicated in the picture with dashed curved lines. This proves (3).


\end{proof}

\begin{prop}\label{prop:centraliff}
Let $T=(\theta_0, \dots, \theta_{k-1}) \in \Con(\A)^k$. Suppose for $\delta \in \Con(\A)$ that $\mathcal{C}(T;l;\delta)$ holds for some $l \in k$. Then $\mathcal{C}(T; i;\delta)$ holds for all $i \in k$.
\end{prop}

\begin{proof}
Choose $j\neq l$. By definition \ref{def:cen}, it suffices to show that for each $h\in M(T)$ that if each $(j)$-supporting line of $h$ is a $\delta$-pair then the $(j)$-pivot line of $h$ is a $\delta$ pair. For $e\in n+1$ consider the $e$th shift rotation at $(j,l)$ of $h$. By $(2)$ of \ref{lem:newop}, each $(l)$-supporting line of $R^e_{j,l}(h)$ is a $\delta$-pair. We assume that $C(T;l;\delta)$ holds, therefore the $(l)$-pivot line of $R^e_{j,l}$ is a $\delta$-pair. Because this is true for every $e \in n+1$, $(3)$ of \ref{lem:newop} shows that the $(j)$-pivot line of $h$ is a $\delta$-pair. We therefore conclude that $C(T;j;\delta)$ holds.
\end{proof}

\begin{thm}
$[T]_j = [T]_l$ for all $j,l \in k$.
\end{thm}

\begin{proof}
$[T]_j = \bigwedge \{\delta : \mathcal{C}(T; j; \delta) \} = \bigwedge \{\delta : \mathcal{C}(T; l; \delta) \} = [T]_l$.

\end{proof}

We can now omit the coordinate $j$ when stating $C(T;j;\delta)$ or referring to $[T]_j$, writing $C(T;\delta)$ and $[T]$ instead.

\section{Generators of Higher Commutator}
In this section we construct for a sequence of congruences $T=(\theta_0, \dots, \theta_{k-1}) \in \Con(\A)^k$ a set of generators $X(T)$ for $[T]$. The idea of the construction is to consider all possible sequences of consecutive shift rotations for an arbitrary $T$-matrix $h$. Each such sequence will produce a $T$-matrix that is constant on all $(k-1)$-supporting lines. The $(k-1)$-pivot line of such a $T$-matrix must belong to any $\delta$ such that $C(T;\delta)$ holds. This is illustrated for $3$-dimensional matrices in Figure \ref{fig:gen1}, where constant pairs are indicated with bold.

\vspace{10mm}

\begin{figure}[!ht]
\begin{center}
\includegraphics{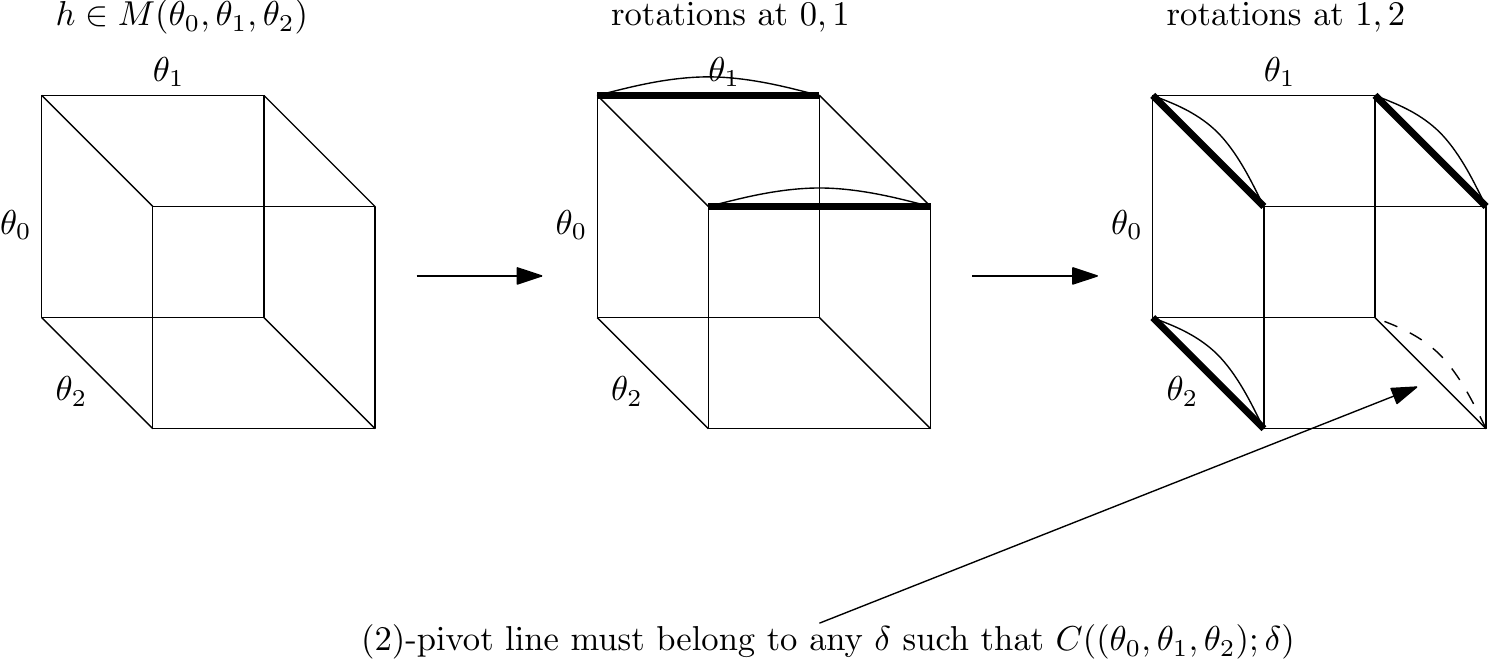}
\end{center}
\caption{Ternary Generators}\label{fig:gen1}
\end{figure}

As usual, let $\var$ be a congruence modular variety with Day terms $m_e$ for $e\in n+1$, and let $T= (\theta_0, \dots, \theta_{k-1}) \in \Con(\A)^k$ for $\A\in \var$. For a $T$-matrix $h$ we will apply a composition of $k-1$ many shift rotations, first at $(0,1)$, then at $(1,2)$, ending at $(k-2,k-1)$. For each stage there are $n+1$ many choices of Day terms, each giving a different shift rotation. It is therefore quite natural to label these sequences of shift rotations with branches belonging to the tree of height $k$ with $n+1$ many successors of each vertex. Set $$\mathbb{D}_{ k} = \langle (n+1)^{<k}; < \rangle, $$
where for $d_1, d_2 \in (n+1)^{<k}$, we have $d_1<d_2$ if $d_2$ extends $d_1$. Note that $\mathbb{D}_{k}$ has the empty sequence $\emptyset$ as a root.

\begin{lem}\label{lem:pairseq}
Let $\var$ be a variety with Day terms $m_e$ for $ e \in n+1$. Let $T= (\theta_0, \dots, \theta_{k-1}) \in \Con(\A)^k$. Let $h\in M(T)$ be labeled by $\tau = (t, \mathcal{P})$. Set $h^\emptyset = h$. For each non-empty $d =(d_0, \dots , d_i) \in \mathbb{D}_{k}$ there is a $T$-matrix $h^d \in M(T)$ labeled by some $\tau^d = (t^d ; \mathcal{P}^d)$ such that 

\begin{enumerate}

\medskip
\item 
$h^d = R_{i, i+1}^{d(i)}(h^c)$, where $c$ is the predecessor of $d$.
\medskip

\item
Let $f \in 2^{k \setminus \{i+1\}}$ be such that $f(j) = 0$ for some $j\in i+1$. Then the $(i+1)$-supporting line of $h^d$ at $f$:
$$ (h^d)_f = \left[  \begin{array}{cc} 
	(t^d)_f(\textbf{a}_{i+1}^d) & (t^d)_f(\textbf{b}_{i+1}^d)\\
				\end{array}
\right]
$$
is a constant pair.

\end{enumerate}

\end{lem}
\begin{proof}
The lemma is trivially true for $h^\emptyset=h$. Suppose it holds for $c$ and let $d$ be a successor of $c$. Let $f \in 2^{k \setminus \{i+1\}}$ be such that $f(j) = 0$ for some $j\in i+1$. We need to establish that the supporting line 
$$ (h^d)_f = \left[  \begin{array}{cc} 
	(t^d)_f(\textbf{a}_{i+1}^d) & (t^d)_f(\textbf{b}_{i+1}^d)\\
				\end{array}
\right]
$$
is a constant pair. Let $f^* = f\vert_{2^{k\setminus \{i, i+1\}}}$ be the restriction of $f$ to $k\setminus \{i, i +1\}$. We treat two cases:

\begin{enumerate}
\item Suppose $j= i$, and for no other $j\in i+1$ does $f(j) = 0$. Consider the $(i, i+1)$-cross-section square of $h^c$ at $f^*$:

$$(h^c)_{f^*} = \left[  \begin{array}{cc} 
	r_{f^*}  & s_{f^*}\\
	u_{f^*} & v_{f^*}\\
				\end{array}
\right]
$$

By \ref{lem:newop}, the $(i, i+1)$-cross-section of $m^d$ at $f^*$ is:
$$
(h^d)_{f^*}= \left[  \begin{array}{cc} 
	s_{f^*}  & s_{f^*}\\
	m_{d(i)}(s_{f^*}, r_{f^*}, u_{f^*}, v_{f^*}) & m_{d(i)}(s_{f^*}, s_{f^*}, v_{f^*}, v_{f^*})\\
				\end{array}
\right]
$$
The $(i+1)$-supporting line of $h^d$ at $f$ is the top row of the above square, that is, 
$$ (h^d)_f = \left[  \begin{array}{cc} 
	s_{f^*}  & s_{f^*}\\
				\end{array}
\right]
$$
\item
Suppose that $f(j)= 0$ for some $j \in i$. In this case the inductive assumption applies to $h^c$, so columns of the $(i, i+1)$-cross-section of $h^c$ at $f^*$ are therefore constant:
$$(h^c)_{f^*} = \left[  \begin{array}{cc} 
	r_{f^*}  & s_{f^*}\\
	r_{f^*} & s_{f^*}\\
				\end{array}
\right]
$$
We therefore compute the $(i, i+1)$-cross-section of $h_{i+1}^d$ at $f^*$ as:

$$
(h^d)_{f^*}= \left[  \begin{array}{cc} 
	s_{f^*}  & s_{f^*}\\
	m_{d(i)}(s_{f^*}, r_{f^*}, r_{f^*}, s_{f^*}) & m_{d(i)}(s_{f^*}, s_{f^*}, s_{f^*}, s_{f^*})\\
				\end{array}
\right] = 
\left[  \begin{array}{cc} 
	s_{f^*}  & s_{f^*}\\
	s_{f^*}  & s_{f^*}\\
				\end{array}
\right]
$$
The $(i+1)$-cross-section line of $h^d$ at $f$ is either the top or bottom row of the above square, if $f(i)=0$ or $f(i)=1$ respectively. Therefore 
$$ (h^d)_f = \left[  \begin{array}{cc} 
	s_{f^*}  & s_{f^*}\\
				\end{array}
\right]
$$
\end{enumerate}
\end{proof}

Let $d = (d_0,\dots, d_{k-2})$ be a leaf of $\mathbb{D}_{k}$. By \ref{lem:pairseq}, all $(k-1)$-supporting lines of $h^d$ are constant pairs $\langle s,s \rangle$.  If we assume that $\mathcal{C}(T;\delta)$ holds then the $(k-1)$-pivot line of $m^d$ must belong to $\delta$. That is, $(h^d)_{\textbf{k-1}} \in \delta$ for any $h\in M(T)$ and any leaf $d \in \mathbb{D}_{k}$. Set 
$$X(T) = \{ (h^d)_{\textbf{k-1}} : h\in M(T), d \in \mathbb{D}_{k} \textnormal{ a leaf } \},$$ see Figure \ref{fig:tree} for a picture.
\vspace{40mm}

\begin{figure}[!ht]

\begin{center}
\includegraphics[scale=.9]{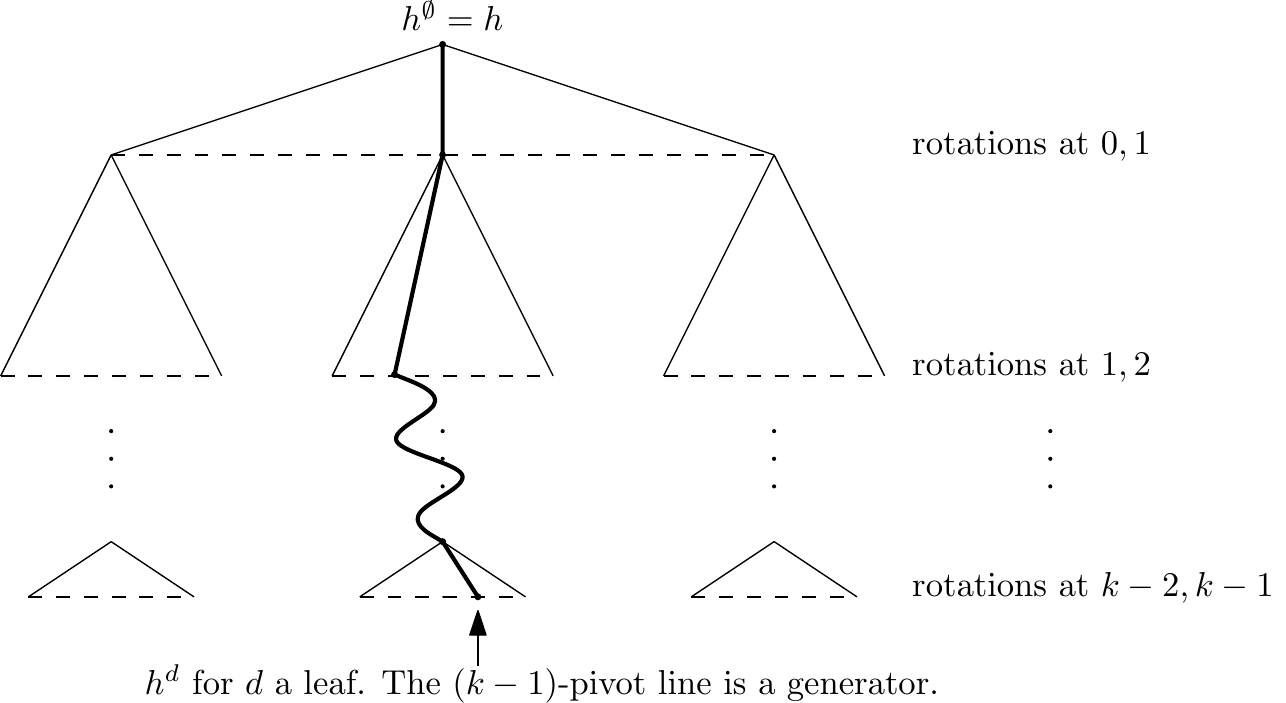}
\end{center}
\caption{Tree}\label{fig:tree}

\end{figure}
\begin{figure}[!ht]
\begin{center}
\includegraphics[scale=.8]{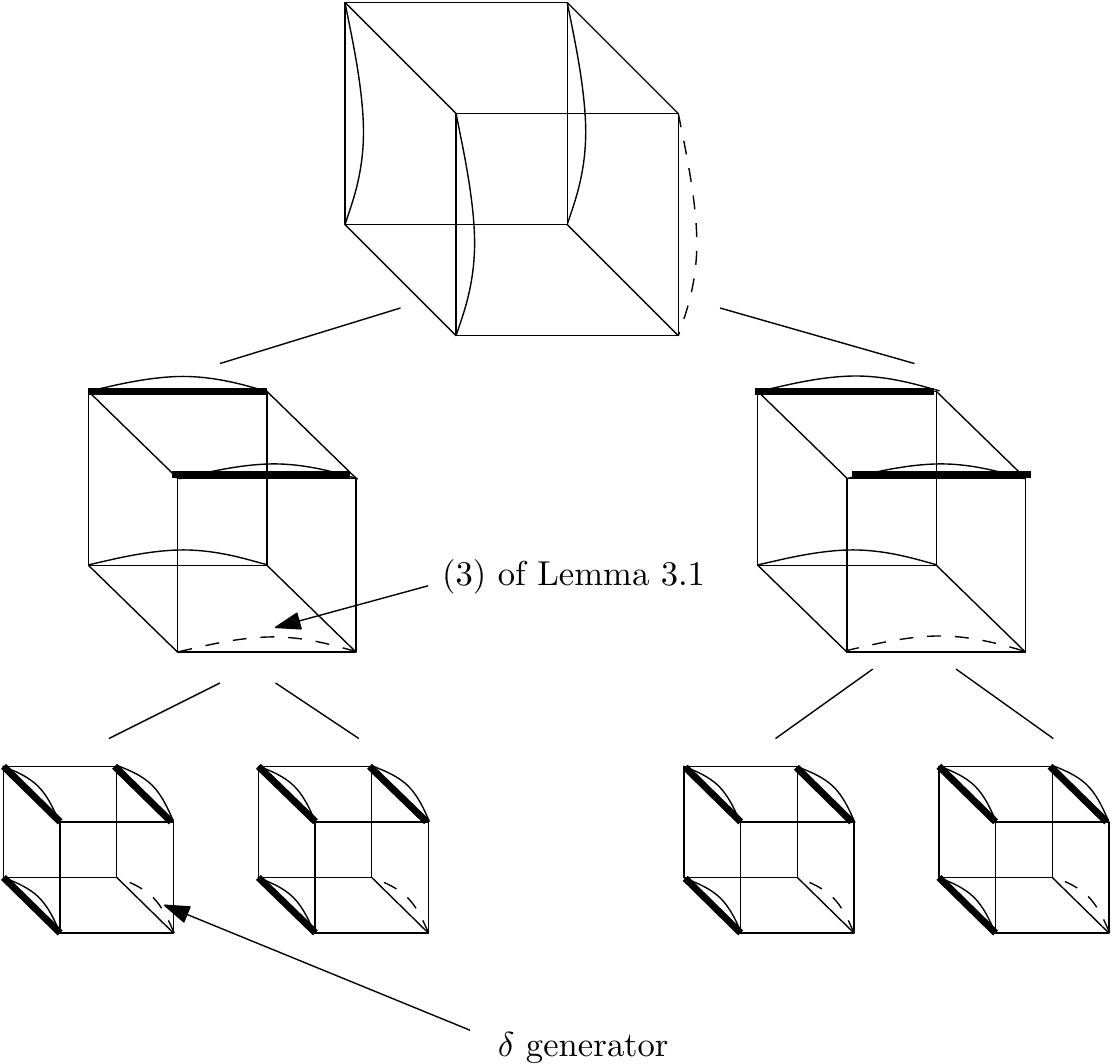}
\end{center}
\caption{Ternary Generator Tree}\label{fig:threetree}

\end{figure}

\newpage

We have just observed that 
\begin{lem}\label{genincen}
Let $T = (\theta_0, \dots, \theta_{k-1}) \in \Con(\A)^k$ for $\A \in \var$, where $\var$ is congruence modular. Suppose that $\delta \in \Con(\A)$ is such that $\mathcal{C}(T; \delta)$ holds. Then $X(T) \subset \delta$. In particular, $\Cg(X(T)) \leq [T]$

\end{lem}

By induction over $\mathbb{D}_{k}$ we now demonstrate the following

\begin{lem}\label{ceningen}
Let $\delta= \Cg(X(T))$. Then $\mathcal{C}(T;\delta)$ holds. In particular, $[T] \leq \Cg(X(T))$. 
\end{lem}
\begin{proof}

Take $h \in M(T)$. By symmetry, it suffices to consider that all $(0)$-supporting lines of $h$ are $\delta$-pairs. We need to show that the $(0)$-pivot line of $h$ is also a $\delta$-pair. By a repeated application of (2) of Lemma \ref{lem:newop}, each $(i+1)$-supporting line of $h^d$ is a $\delta$-pair, where $d = (d_0, \dots, d_i) \in \mathbb{D}_{k}$. Take $c = (c_0, \dots, c_{i-1}) \in \mathbb{D}_{k}$, and suppose that for all successors $d = (c_0,\dots, c_{i-1}, d_i)$ of $c$ that the $(i+1)$-pivot line of $h^d$ is a $\delta$-pair. Applying (3) of Lemma \ref{lem:newop} yields that the $(i)$-pivot line of $h^c$ is a $\delta$-pair. Because $\delta = \Cg(X(T))$, the $(k-1)$-pivot line of $h^d$ is a $\delta$-pair for any $d \in \mathbb{D}_{k}$ that is a leaf. By induction it follows that the $(0)$-pivot line of $h$ is a $\delta$-pair, as desired. See Figure \ref{fig:threetree} for a picture.

\end{proof}

\begin{thm}
The following hold:
\begin{enumerate}
\item
$[T] = \Cg(X(T))$
\item
$\mathcal{C}(T;\delta)$ if and only if $[T] \leq \delta$
\end{enumerate}
\end{thm}

\begin{proof}
This follows from Lemmas \ref{genincen} and \ref{ceningen}.
\end{proof}

\section{Additivity and Homomorphism Property}
We are now ready to show that the commutator is additive and is preserved by surjections. We begin by example, demonstrating additivity for the $3$-ary commutator. Let $\theta_0, \theta_1, \gamma_i (i \in I)$ be a collection of congruences of $\A$. We want to show that $[\theta_0, \theta_1, \Join_{i\in I} \gamma_i] = \Join_{i\in I} [\theta_0, \theta_1, \gamma_i]$. It is immediate that $[\theta_0, \theta_1, \Join_{i\in I} \gamma_i] \geq \Join_{i\in I} [\theta_0, \theta_1, \gamma_i]$, because of monotonicity. To demonstrate the other direction, it suffices to show that $C((\theta_0, \theta_1, \Join_{i\in I} \gamma_i); \alpha)$ holds, where $\alpha = \Join_{i\in I} [\theta_0, \theta_1, \gamma_i]$. 

Let $h \in M(\theta_0, \theta_1, \Join_{i\in I} \gamma_i)$ be labeled by $\tau = (t(\textbf{z}_0, \textbf{z}_1, \textbf{z}_2), ((\textbf{a}_0, \textbf{b}_0)), (\textbf{a}_1, \textbf{b}_1), (\textbf{a}_2, \textbf{b}_2)))$. Suppose that each $(0)$-supporting line of $h$ is an $\alpha$-pair. We need to show that the $(0)$-pivot line of $h$ is also an $\alpha$-pair.

Because $\textbf{a}_2 \equiv_{\Join_{i\in I} \gamma_i} \textbf{b}_2$, there exist tuples $\textbf{c}_0, \dots, \textbf{c}_q$ such that 
$$\textbf{a}_2 = \textbf{c}_0 \equiv_{\gamma_{i_0}} \textbf{c}_1 \dots \textbf{c}_{q-2}\equiv_{\gamma_{q-1}} \textbf{c}_q = \textbf{b}_2$$

This sequence of tuples produces the sequence of cross-section squares shown in Figure \ref{fig:add1}. Each square is a $(\theta_0, \theta_1)$-matrix labeled by $(t(\textbf{z}_0, \textbf{z}_1, \textbf{c}_s), ((\textbf{a}_0, \textbf{b}_0)), (\textbf{a}_1, \textbf{b}_1)))$, for $\textbf{c}_s$ a tuple from $\textbf{c}_0, \dots, \textbf{c}_q$. Each consecutive pair of squares labeled by $$(t(\textbf{z}_0, \textbf{z}_1, \textbf{c}_s), ((\textbf{a}_0, \textbf{b}_0)), (\textbf{a}_1, \textbf{b}_1))) \textnormal{ and }(t(\textbf{z}_0, \textbf{z}_1, \textbf{c}_{s+1}), ((\textbf{a}_0, \textbf{b}_0)), (\textbf{a}_1, \textbf{b}_1)))$$
are the $2$-cross-section squares of a $(\theta_0, \theta_1, \gamma_{i_s})$-matrix. As usual, $\alpha$-pairs are indicated with curved lines.

\begin{figure}[!ht]
\begin{center}
\includegraphics{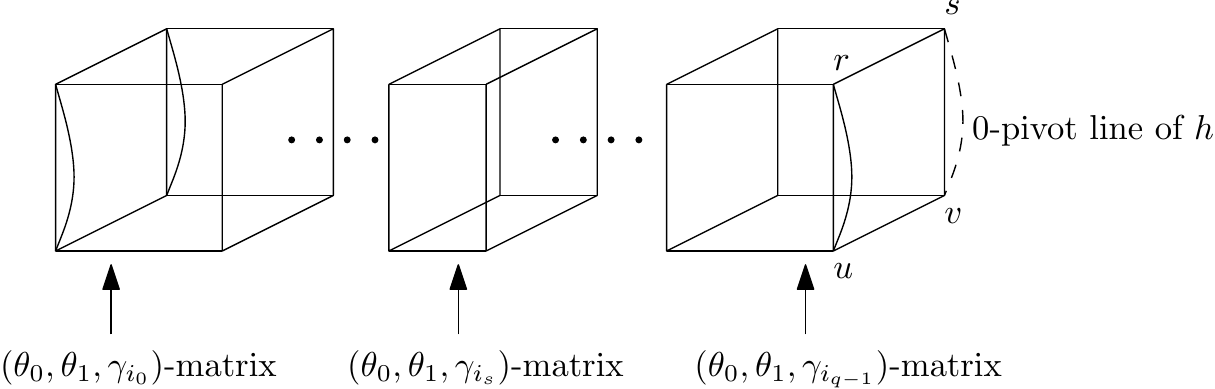}
\end{center}
\caption{Sequence of Matrices}\label{fig:add1}
\end{figure}

To show that the $(0)$-pivot line of $h$ is an $\alpha$-pair it suffices to show that $$\langle m_e(s,s,v,v), m_e(s,r,u,v) \rangle \in \alpha $$ for all $e \in n+1$. Therefore, we consider for each $e \in n+1$ the $e$-th shift rotation at $(0,1)$ of the above sequence of matrices. This is shown in Figure \ref{fig:add2}. Constant pairs are indicated with bold.

\begin{figure}
\begin{center}
\includegraphics{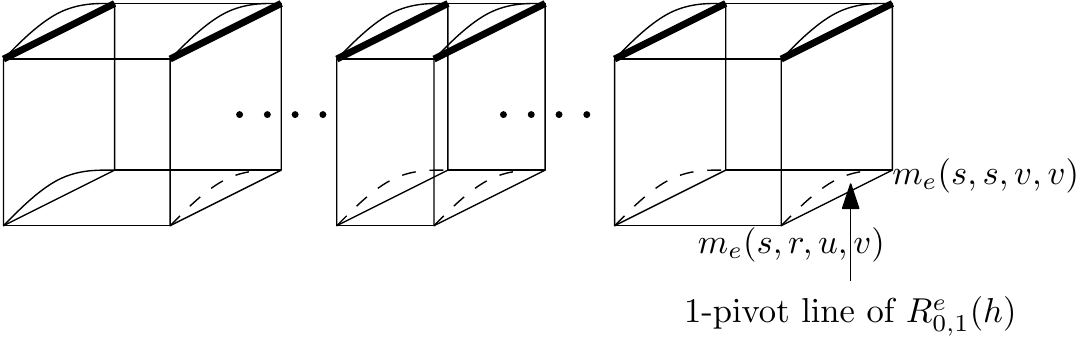}
\end{center}
\caption{Rotated Sequence}\label{fig:add2}
\end{figure}

Because $[\theta_0, \theta_1, \gamma_i] \leq \alpha$ for all $i \in I$, we have that $C(\theta_0, \theta_1, \gamma_i; \alpha)$ holds. Because each cube in the above sequence is a $(\theta_0, \theta_1, \gamma_i)$-matrix for some $i\in I$, it follows by induction that $(1)$-pivot line of $R^e_{0,1}(h)$ is an $\alpha$-pair, as desired. 

To show the additivity of a commutator of any arity, the same argument is used. For $h \in M(T)$ we consider all $h^d$ for any $d \in \mathbb{D}_k$ that is a predecessor of a leaf. By \ref{lem:pairseq}, all $(k-2)$-supporting lines that do not belong to the $(k-2,k-1)$-pivot square of $h^d$ are constant pairs. The argument is then essentially the same as the $3$-ary example above, complicated slightly by an induction over the tree $\mathbb{D}_k$.

\begin{prop}\label{additivity}
Let $\gamma_i $ for $i\in I$ be a collection of congruences of $\A$. Set $T=( \theta_0, ... , \theta_{k-1}, \Join_{i\in I} \gamma_i )$ and $T_i = (\theta_0, ... , \theta_{k-1}, \gamma_i )$, where $\theta_0, ... , \theta_{k-1} \in \Con(\A)$. Then $[T]= \Join_{i\in I} [T_i]$.
\end{prop}
\begin{proof}

By monotonicity, $\Join_{i\in I} [T_i]  \leq [T]$. Set $\alpha = \Join_{i\in I} [T_i]$. We need to show that $C(T; \alpha)$ holds. Let $h\in M(T)$ be labeled by $\tau = (t(\textbf{z}_0, \dots, \textbf{z}_k), \mathcal{P})$, where $\mathcal{P}$ is a sequence of pairs of tuples $((\textbf{a}_0, \textbf{b}_0), \dots, (\textbf{a}_k, \textbf{b}_k))$. Suppose that every $(0)$-supporting line of $h$ is a $\alpha$-pair. We will show that the $(0)$-pivot line of $h$ is an $\alpha$-pair also. 

Here we have that $\textbf{a}_k \equiv_{\Join_{i\in I} \gamma_i} \textbf{b}_k$. We illustrate the $(k+1)$-dimensional matrix $h$ as the product of two $k$-dimensional matrices in Figure \ref{fig:add3}, given by evaluating $\textbf{z}_k$ at either $\textbf{a}_k$ or $\textbf{b}_k$. These two matrices are called $\eta_0$ and $\eta_1$ respectively.

\begin{figure}[!ht]
\begin{center}
\vspace{10mm}
\includegraphics{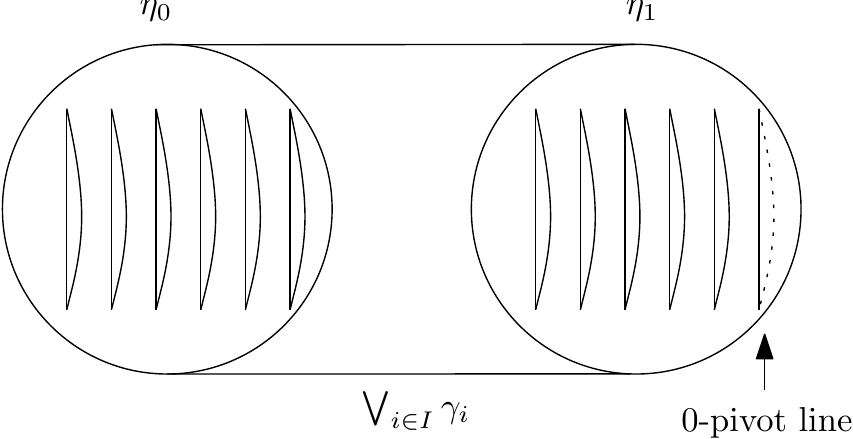}
\vspace{10mm}

\includegraphics{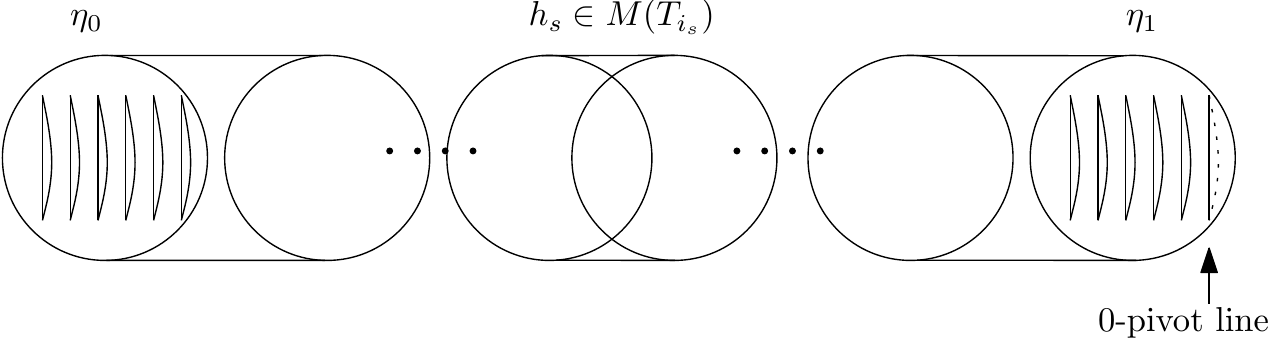}
\vspace{10mm}
\includegraphics{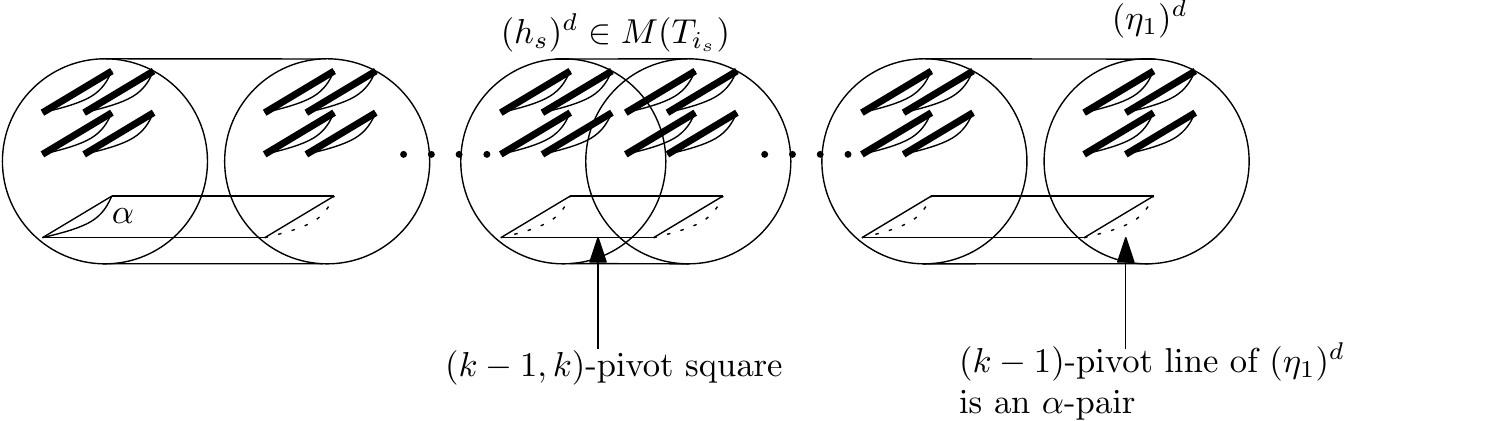}
\end{center}
\caption{Sequence of Matrices and Rotations}\label{fig:add3}
\end{figure}

Notice that the $(0)$-pivot line of $h$ is equal to the $(0)$-pivot line of $\eta_1$. By an induction identical to that given in the proof of Lemma \ref{ceningen} it therefore suffices to show that the $(k-1)$-pivot line of $(\eta_1)^d$ is an $\alpha$-pair, for each $d\in \mathbb{D}_k$ that is a leaf.

Because $\textbf{a}_k \equiv_{\Join_{i\in I} \gamma_i} \textbf{b}_k$, there exist tuples $\textbf{c}_0, \dots, \textbf{c}_q$ such that 
$$\textbf{a}_k = \textbf{c}_0 \equiv_{\gamma_{i_0}} \textbf{c}_1 \dots \textbf{c}_{q-2}\equiv_{\gamma_{q-1}} \textbf{c}_q = \textbf{b}_k$$
Evaluating $\textbf{z}_k$ at each of the $\textbf{c}_s$ gives the sequence of matrices shown in Figure \ref{fig:add3}, where each consecutive pair of matrices corresponding to the tuples $\textbf{c}_s, \textbf{c}_{s+1}$ forms a $T_{i_s}$-matrix which we call $h_{s}$.

Now, take $d \in \mathbb{D}_{k}$ to be a leaf. Notice that $d\in \mathbb{D}_{k+1}$ and that $d$ is a predecessor of a leaf in this tree. For each $h_{i_s}$ in the above sequemce, consider the $T_{i_s}$-matrix $(h_{i_s})^d$. This gives the final sequence of matrices shown in Figure \ref{fig:add3}. By Lemma \ref{lem:pairseq}, every $(k-1)$-supporting line that does not belong to a $(k-1,k)$-pivot square is a constant pair. These are drawn with bold. The sequence of $(k-1,k)$-pivot squares is drawn underneath the constant supporting lines.

As in the 3-dimensional example, we observe that $C(T_i; \alpha)$ holds. It follows from induction that the $(k-1)$-pivot line of $(\eta_1)^d$ is an $\alpha$-pair, as desired.

\end{proof}

Let $f: \A \rightarrow \mathbb{B}$ be a surjective homomorphism with kernel $\pi$. Abusing notation, we denote by $T \join \pi$ the sequence of congruences $(\theta_1 \join \pi, \dots , \theta_k \join \pi )$, and by $f(T)$ the sequence of congruences $(f(\theta_1), \dots,  f(\theta_k) )$. We then have the following 

\begin{prop}
Let $f: \A \rightarrow \mathbb{B}$ be a surjective homomorphism with kernel $\pi$. Then $[T] \join \pi = f^{-1}([f(T \join \pi)])$.

\end{prop}
\begin{proof}
We argue by generators again. By Proposition \ref{additivity} we have that $[T] \join \pi = [T \join \pi] \join \pi$. So, we assume without loss that $\theta_i \geq \pi$ for $1\leq i \leq k$. Notice that $[T] \join \pi = \Cg(X(T) \union \pi)$ and that $f(X(T) \union \pi) = X(f(T))$. But $[f(T)] = \Cg(X(f(T))$, so $f$ carries a set of generators for $[T] \join \pi$ onto a set of generators for $[f(T)]$. Therefore $f([T] \join \pi) = [f(T)]$ as desired.
\end{proof}

\section{Two Term Commutator}

Kiss showed in \cite{threeremarks} that the term condition definition of the binary commutator is equivalent to a commutator defined with a two term condition. The method of proof uses a difference term. We begin this section be examining the binary case. The equivalence of the commutator defined with the term condition to the commutator defined with a two term condition can be shown using Day terms. This approach easily generalizes to the higher commutator. Recall that for a matrix $h \in M(\theta_0, \dots, \theta_{k-1})$ and $f\in 2^k$ we denote by $h_f$ the vertex of $h$ that is indexed by $f$.

\begin{defn}{(Binary Two Term Centralization)}
Let $\var$ be a congruence modular variety and take $\A \in \var$. For $\alpha, \beta, \delta \in \Con(\A)$ we say that \textbf{$\alpha$ two term centralizes $\beta$ modulo $\delta$} if the following condition holds for all $h,g  \in M(\alpha, \beta)$, where we assume $h$ and $g$ are respectively labeled by $(t(\textbf{z}_0, \textbf{z}_1), ((\textbf{a}_0, \textbf{b}_0), (\textbf{a}_1, \textbf{b}_1))) \textnormal{ and }$
$(s(\textbf{x}_0, \textbf{x}_1), ((\textbf{c}_0, \textbf{d}_0), (\textbf{c}_1, \textbf{d}_1))):$

\begin{enumerate}
\item[] $\langle s(\textbf{c}_0, \textbf{c}_1), t(\textbf{a}_0, \textbf{a}_1) \rangle \in \delta$,
\item[]$\langle s(\textbf{c}_0, \textbf{d}_1), t(\textbf{a}_0, \textbf{b}_1) \rangle \in \delta$,
\item[] $\langle s(\textbf{d}_0, \textbf{c}_1), t(\textbf{b}_0, \textbf{a}_1) \rangle \in \delta$ imply
\item[] $\langle s(\textbf{d}_0, \textbf{d}_1), t(\textbf{b}_0, \textbf{b}_1) \rangle \in \delta$.

\end{enumerate}

This condition is abbreviated as $C_{tt}(\alpha, \beta)$.
\end{defn}

Figure \ref{fig:twoterm1} depicts the condition $C_{tt}(\alpha, \beta)$. Curved lines represent $\delta$-pairs. The top matrix is labeled by $$(t(\textbf{z}_0, \textbf{z}_1), ((\textbf{a}_0, \textbf{b}_0), (\textbf{a}_1, \textbf{b}_1)))$$
and the bottom matrix is labeled by 
$$(s(\textbf{x}_0, \textbf{x}_1), ((\textbf{c}_0, \textbf{d}_0), (\textbf{c}_1, \textbf{d}_1)))$$

\medskip

\begin{figure}[!ht]
\begin{center}
\includegraphics[scale=.9]{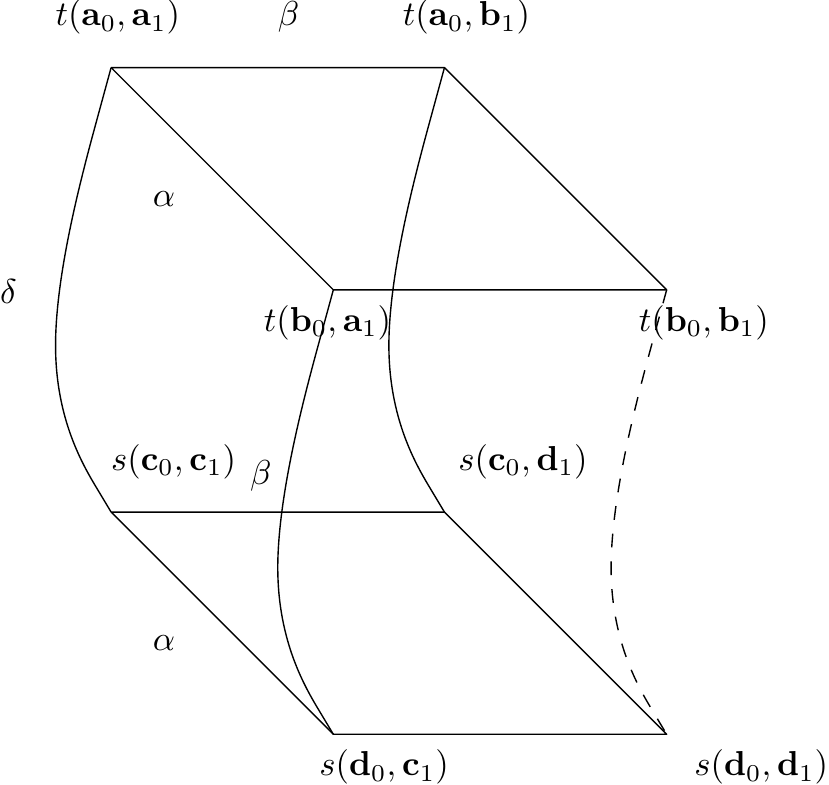}
\end{center}
\caption{Binary Two Term Condition}\label{fig:twoterm1}
\end{figure}

\begin{prop}
$C(\alpha, \beta; \delta)$ holds if and only if $C_{tt}(\alpha, \beta; \delta)$ holds.

\end{prop}

\begin{proof}

Suppose $C_{tt}(\alpha, \beta; \delta)$ holds. To show that $C(\alpha, \beta; \delta)$ holds we take $\left[ \begin{array}{cc}
				a&b\\
				c&d\\
				\end{array}
				\right] \in M(\alpha, \beta)$ such that $\langle a, c \rangle \in \delta$. Figure \ref{fig:twoterm2} demonstrates that if $C_{tt}(\alpha, \beta)$ holds then $\langle b,d \rangle \in \delta$.

\begin{figure}[!ht]
\begin{center}
\includegraphics[scale=.9]{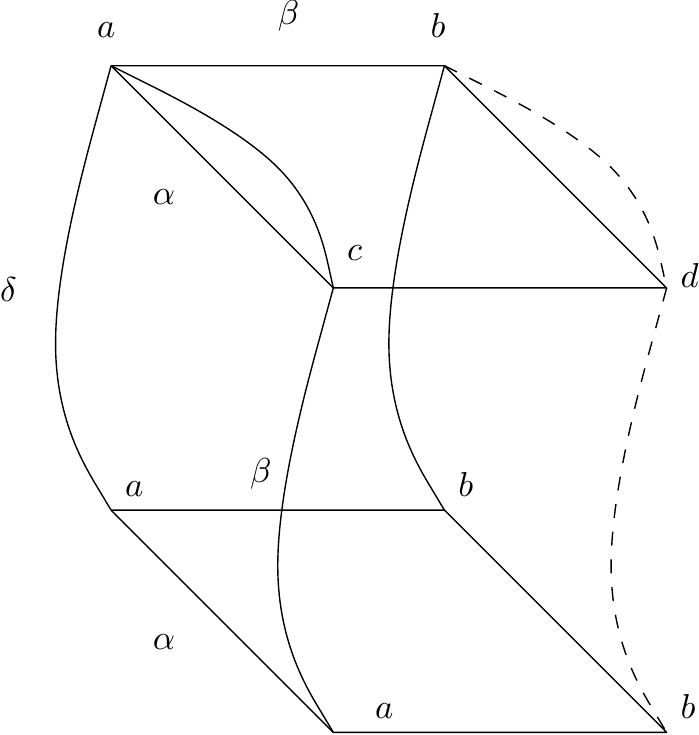}
\caption{$C_{tt}(\alpha, \beta; \delta) $ implies $C(\alpha, \beta; \delta)$}\label{fig:twoterm2}
\end{center}
\end{figure}

Suppose now that $C(\alpha, \beta; \delta)$ holds. Let $g,h \in M(\alpha, \beta)$ be labeled by

 $(s(\textbf{z}_0, \textbf{z}_1), ((\textbf{a}_0, \textbf{b}_0), (\textbf{a}_1, \textbf{b}_1)))$ and 
$(t(\textbf{x}_0, \textbf{x}_1), ((\textbf{c}_0, \textbf{d}_0), (\textbf{c}_1, \textbf{d}_1)))$
respectively. Suppose that 
\begin{enumerate}
\item $\langle s(\textbf{a}_0, \textbf{a}_1), t(\textbf{c}_0, \textbf{c}_1) \rangle = \langle a,e \rangle \in \delta$
\item $\langle s(\textbf{b}_0, \textbf{a}_1), t(\textbf{d}_0, \textbf{c}_1) \rangle = \langle  b,f \rangle \in \delta$
\item $\langle s(\textbf{a}_0, \textbf{b}_1), t(\textbf{c}_0, \textbf{d}_1) \rangle= \langle  c,g \rangle  \in \delta$

\end{enumerate}
We need to show that  $\langle s(\textbf{d}_0, \textbf{d}_1), t(\textbf{b}_0, \textbf{b}_1) \rangle = \langle d,h \rangle \in \delta$.

 We construct a matrix that is similar to a shift rotation. For each $e \in n+1$ consider the polynomial 
$$p_e(\textbf{y}_0, \textbf{y}_1) = m_e(t(\textbf{y}_0^0, \textbf{y}_1^0), t(\textbf{y}_0^1, \textbf{y}_1^0), s(\textbf{y}_0^2, \textbf{y}_1^1), s(\textbf{y}_0^3, \textbf{y}_1^1))$$

where $\textbf{y}_0 = \textbf{y}_0^0\concat \textbf{y}_0^1 \concat \textbf{y}_0^2\concat \textbf{y}_0^3$ and $\textbf{y}_1 = \textbf{y}_1^0\concat \textbf{y}_1^1$.

Set 
\begin{enumerate}
\item
$\textbf{u}_0 = \textbf{b}_0 \concat \textbf{b}_0 \concat \textbf{d}_0 \concat \textbf{d}_0$
\item
$\textbf{v}_0 = \textbf{b}_0 \concat \textbf{a}_0 \concat \textbf{c}_0 \concat \textbf{d}_0$

\item
$\textbf{u}_1 = \textbf{a}_1 \concat \textbf{c}_1$

\item
$\textbf{v}_1 = \textbf{b}_1 \concat \textbf{d}_1$

\end{enumerate}

Let $ q_e \in M(\alpha, \beta)$ be labeled by $(p_e, ((\textbf{u}_0, \textbf{v}_0), (\textbf{u}_1, \textbf{v}_1)))$. The relationship between $h,g$ and $q_e$ is shown in Figure $\ref{fig:twoterm3}$.

\begin{figure}[!ht]
\begin{center}
\includegraphics[scale=1]{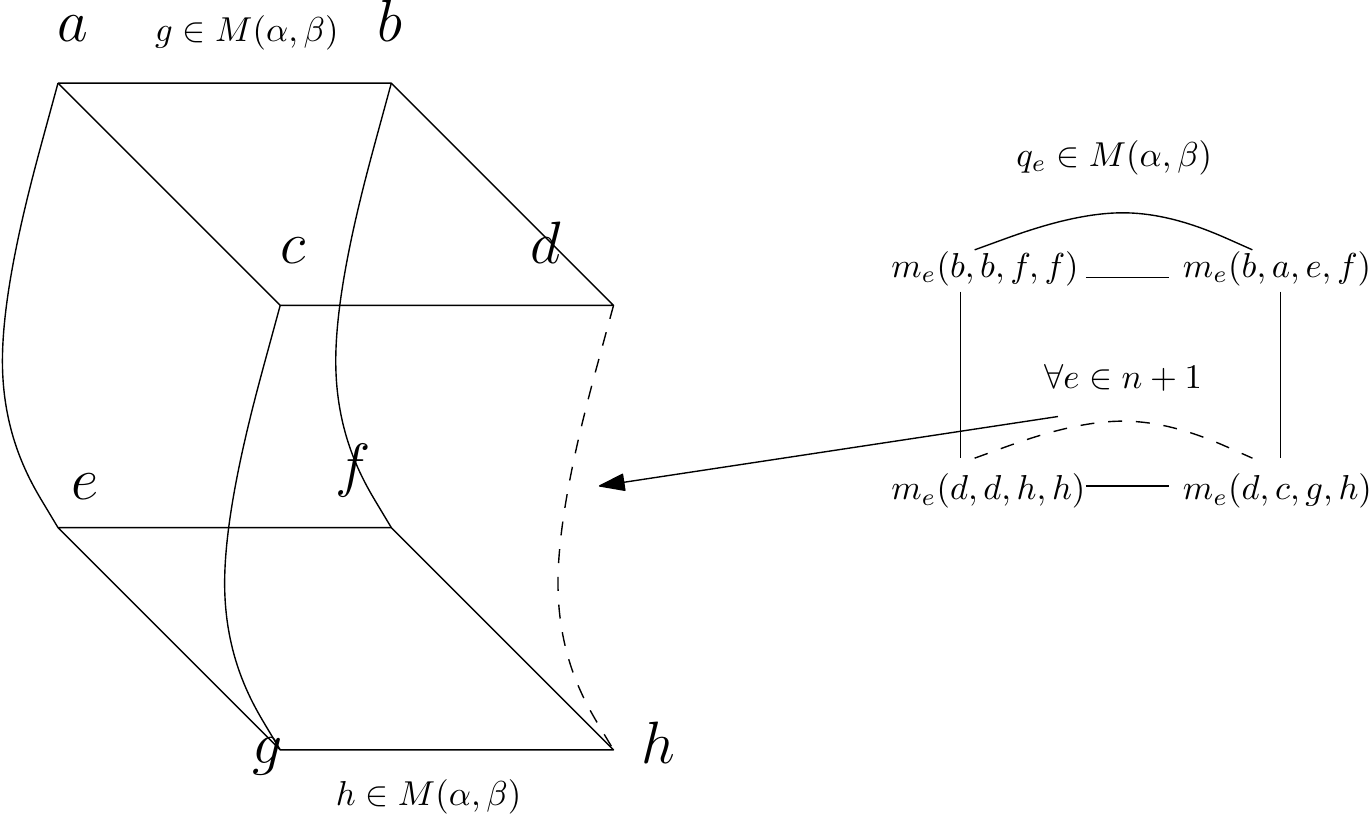}
\end{center}
\caption{$C(\alpha, \beta; \delta)$ implies $C_{tt}(\alpha, \beta; \delta) $  }\label{fig:twoterm3}
\end{figure}
Proposition \ref{prop:shift} show that $\langle m_e(b,a,e,f), m_e(b,b,f,f) \rangle \in \delta$ because $\langle a, e \rangle $ and $\langle b, f \rangle$ are $\delta$-pairs. We assume that $C(\alpha, \beta; \delta)$ holds, so $\langle m_e(d,c,g,h), m_e(d,d,g,h) \rangle \in \delta$. This holds for all $e \in n+1$ so applying Proposition \ref{prop:shift} again shows that $\langle d,h \rangle \in \delta$.  

\end{proof}

We now generalize this notion to the higher commutator.
\begin{defn}[Two Term Centralization]
Let $\var$ be a congruence modular variety and take $\A \in \var$. For $T = (\theta_0, \dots, \theta_{k-1}) \in \Con(\A)^k$ and $\delta \in \Con(\A)$ we say that \textbf{$T$ is two term centralized at $j$ modulo $\delta$} if the following condition holds for all $h, g \in M(T)$:

\begin{enumerate}
\item If $h_f \equiv_\delta g_f$ for all $f \in 2^k$ except the function that takes constant value $1$ then $h_f \equiv g_f$ for all $f\in 2^k$

\end{enumerate}

This condition is abbreviated as $C_{tt}(T; \delta)$.
\end{defn}

\begin{prop}
$C(T; \delta)$ holds if and only if $C_{tt}(T; \delta)$ holds.

\end{prop}
\begin{proof}
Suppose $C_{tt}(T; \delta)$ holds. To show that $C(T; \delta)$ holds, take $h \in M(T)$ with $0$-supporting lines $\langle a_i, b_i \rangle$ for $ i \in 2^{k-1}-1$ and $0$-pivot line $\langle c, d \rangle$. Suppose that each $0$-supporting line $\langle a_i, b_i \rangle$ is a $\delta$-pair. There is a $g \in M(T)$ with $0$-supporting lines $\langle a_i, a_i \rangle $ for $i \in 2^{k-1}-1$ and $0$-pivot line $\langle c, c \rangle$. We have that $h_f \equiv_\delta g_f$ for all $f \in 2^k$ except possibly the constant function with value $1$. The assumption that $C_{tt}(T; \delta)$ implies that $h_f \equiv_\delta g_f$ for all $f \in 2^k$. In particular, $\langle c,d \rangle  \in \delta$. This is shown in Figure \ref{fig:twoterm4}.

\begin{figure}[!ht]
\begin{center}
\includegraphics{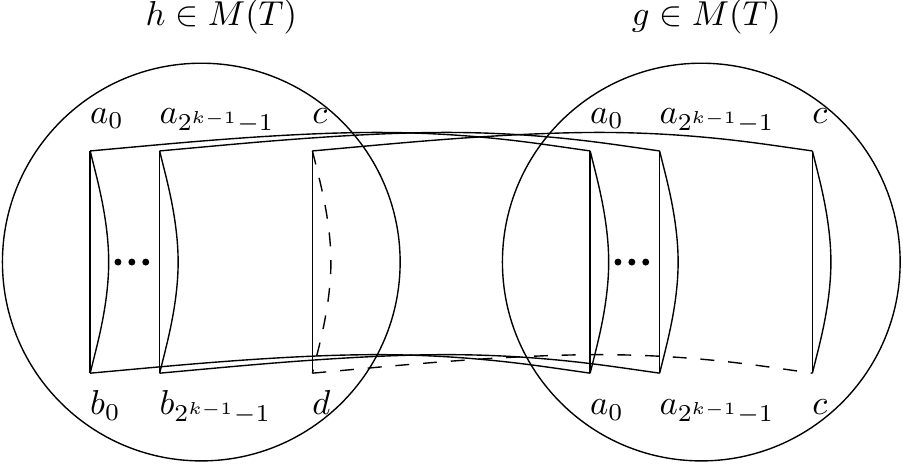}
\end{center}
\caption{$C_{tt}(T; \delta)$ implies $C(T; \delta)$}\label{fig:twoterm4}
\end{figure}

Suppose now that $C(T; \delta)$ holds. Take $h, g\in M(T)$ such that $h_f \equiv_\delta g_f$ for all $f \in 2^k$ except the the function that takes constant value $1$. We want to show that $h_f \equiv g_f$ for all $f\in 2^k$. 

Suppose that $h$ and $g$ are labeled by $$(t(\textbf{z}_0, \dots,  \textbf{z}_{k-1}), ((\textbf{a}_0, \textbf{b}_0), \dots,  (\textbf{a}_{k-1}, \textbf{b}_{k-1}))) \textnormal{ and }$$
$$(s(\textbf{x}_0, \dots,  \textbf{x}_{k-1}), ((\textbf{c}_0, \textbf{d}_0), \dots (\textbf{c}_{k-1}, \textbf{d}_{k-1})))$$
respectively. Choose $i \in k$. Figure \ref{fig:twoterm5} shows the $i$-cross-section lines of $h$ and $g$, with vertices that are $\delta$-pairs connected by curved lines.

\begin{figure}[!ht]
\begin{center}
\includegraphics{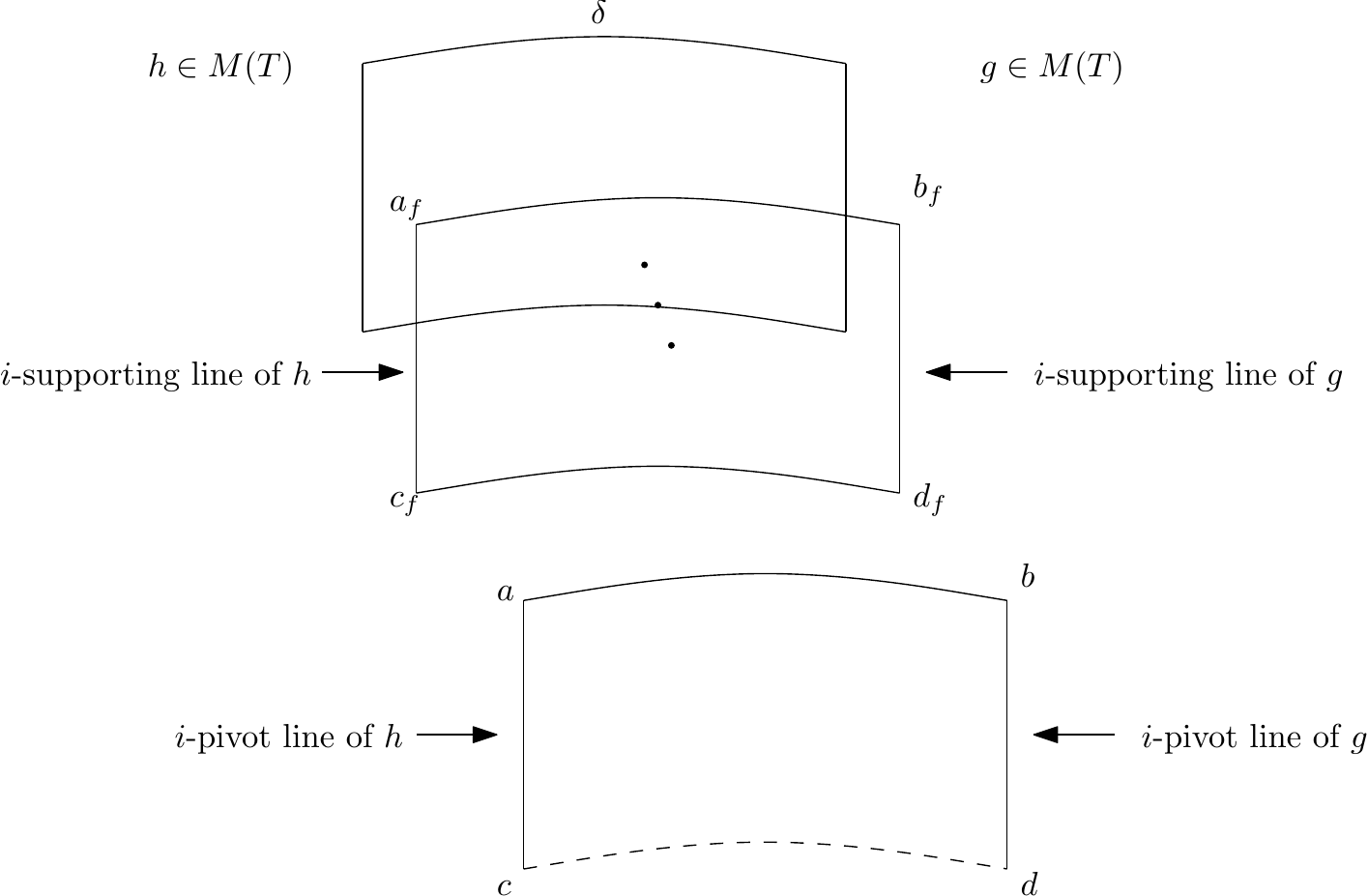}
\end{center}
\caption{$C(T; \delta)$ implies $C_{tt}(T; \delta)$ }\label{fig:twoterm5}
\end{figure}

\begin{figure}[!ht]
\begin{center}
\includegraphics[scale=1]{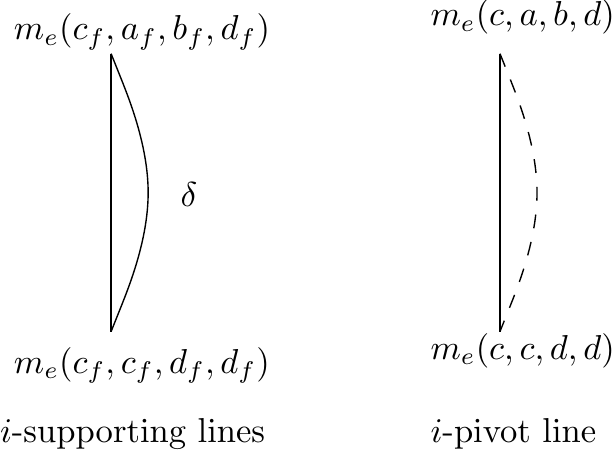}
\end{center}
\caption{Supporting and Pivot Lines}\label{fig:twoterm6}
\end{figure}

We label the $i$-pivot line of $h$ as the pair $\langle a, c \rangle$ and the $i$-pivot line of $g$ as the pair $\langle b,d \rangle$. For a function $f \in 2^{k\setminus \{i\}}$ the supporting lines $h_f$ and $g_f$ are named $\langle a_f, c_f \rangle $ and $\langle b_f, d_f \rangle$ respectively. We want to show that $\langle c, d \rangle \in \delta$. By Proposition \ref{prop:shift}, it suffices to show that $\langle m_e(c,a,b,d) , m_e(c,c,d,d) \rangle \in \delta$ for all $e \in n+1$. This will follow from the assumption that $C(T; \delta)$ holds and the existence of a $T$-matrix $q_e$ with the $i$-cross-section lines shown in Figure \ref{fig:twoterm6}.

Indeed, for each $e \in n+1$ consider the polynomial $p_e(\textbf{y}_0, \dots,  \textbf{y}_{k-1}) = $
$$m_e\bigg(t(\textbf{y}_0^0, \dots, \textbf{y}_i^0,\dots \textbf{y}_{k-1}^0), t(\textbf{y}_0^0, \dots, \textbf{y}_i^1, \dots \textbf{y}_{k-1}^0), $$ $$s(\textbf{y}_0^1, \dots, \textbf{y}_i^2,\dots \textbf{y}_{k-1}^1), s(\textbf{y}_0^1, \dots, \textbf{y}_i^3, \dots \textbf{y}_{k-1}^1)\bigg)$$
where $\textbf{y}_i = \textbf{y}_i^0\concat \textbf{y}_i^1 \concat \textbf{y}_i^2\concat \textbf{y}_i^3$ and $\textbf{y}_j = \textbf{y}_j^0\concat \textbf{y}_j^1$ for $j \neq i$.

Set 
\begin{enumerate}
\item[]
$\textbf{u}_i = \textbf{b}_i \concat \textbf{b}_i \concat \textbf{d}_i \concat \textbf{d}_i$
\item[]
$\textbf{v}_i = \textbf{b}_i \concat \textbf{a}_i \concat \textbf{c}_i \concat \textbf{d}_i$
\end{enumerate}
and for $j \neq i$ 

\begin{enumerate}
\item[]
$\textbf{u}_j = \textbf{a}_j \concat \textbf{c}_j$

\item[]
$\textbf{v}_j = \textbf{b}_j \concat \textbf{d}_j$

\end{enumerate}

Let $ q_e \in M(T)$ be labeled by $(p_e, ((\textbf{u}_0, \textbf{v}_0), \dots,  (\textbf{u}_{k-1}, \textbf{v}_{k-1})))$. By Proposition \ref{prop:shift}, every $i$-supporting line of $q_e$ is a $\delta$-pair. We assume that $C(T; \delta)$ holds, so the $i$-pivot line of $q_e$ is a $\delta$-pair. This holds for all $e \in n+1$, so $\langle c, d \rangle \in \delta$ as desired.

\end{proof}

\bibliographystyle{plain}	
\bibliography{refs}		

\end{document}